\documentclass[a4paper]{article}

\usepackage[utf8]{inputenc}             
\usepackage[T1]{fontenc}                
\usepackage[english]{babel}              
\usepackage{graphicx}                   
\usepackage{amsmath,amsfonts,amssymb}   
\usepackage{minted}                   
\usepackage[all]{xy}
\usepackage{amsthm}
\usepackage{bm}
\usepackage{csquotes}
\usepackage[bottom=3cm, top=3cm, left=3cm, right=3cm]{geometry}
\usepackage{algorithmic}
\usepackage{empheq}
\usepackage{hyperref}
\usepackage{appendix}
\usepackage{subcaption}
\usepackage{color}
\usepackage{enumitem}
\usepackage{authblk}
\usepackage{lmodern} 
\usepackage{bm}
\usepackage{mathtools}
\usepackage{nccmath}
\usepackage{ stmaryrd }
\usepackage{soul}
\usepackage[dvipsnames]{xcolor}
\usepackage[linesnumbered,ruled,vlined]{algorithm2e}
\usepackage[nameinlink, capitalise]{cleveref}
\usepackage[
    backend=biber,
    uniquename=false,
    bibencoding=utf8,
    sorting=nyt,
    doi=false, isbn=false, url=false,
    style=alphabetic,
    maxcitenames=2,
    uniquelist=false,
    maxbibnames=123,
    backref=false,  
    hyperref=true
]{biblatex}
\usepackage{booktabs}
\usepackage{adjustbox} 
\usepackage{multirow}
\usepackage{wrapfig}
\usepackage{threeparttable} 
\DeclareMathOperator{\R}{\mathbb{R}}

\DeclareMathOperator{\id}{id}

\newcommand{\Diff}{\operatorname{Diff}}

\newcommand{\Supp}{\operatorname{Supp}}
\newcommand{\Aff}{\operatorname{Aff}}
\newcommand{\aff}{\mathfrak{aff}}

\newcommand{\app}[4]{\begin{array}{ccl}
   #1 & \longrightarrow & #2 \\
   #3 & \longmapsto & #4 \\
\end{array}}

\newcommand{\D}[1]{\mathsf{D}_{#1}}

\newtheorem{definition}{Definition}[section]
\newtheorem{proposition}[definition]{Proposition}

\newtheorem{remark}[definition]{Remark}

\addto\captionsfrench{}
\hypersetup{
    colorlinks=true,
    linktoc=all,
    linkcolor=blue,   
    citecolor=red,    
    urlcolor=red,
    bookmarksnumbered=true,
    unicode=true
}
\PassOptionsToPackage{unicode}{hyperref}
\PassOptionsToPackage{naturalnames}{hyperref}

\title{A Framework for Joint Affine and Diffeomorphic Image Registration} 

\author[1]{Anton François\thanks{\texttt{anton.francois@ens-paris-saclay.fr}}}
\author[2]{Rayane Mouhli\thanks{\texttt{rayane.mouhli@math.cnrs.fr}}}
\author[1]{Thomas Pierron \thanks{\texttt{thomas.pierron@ens-paris-saclay.fr}}}

\affil[1]{ENS Paris-Saclay, Centre Borelli}
\affil[2]{Université Paris Cité, MAP5 \& Sorbonne Université, LJLL}

\date{}

\begin{document}
\maketitle

\begin{abstract}
    Anatomical image registration commonly relies on a sequential pipeline where an affine alignment is estimated first and then held fixed while a non-rigid diffeomorphic deformation is applied. This two-step process often leads to suboptimal results, as the initial stage can absorb local deformations, biasing the residual passed to the diffeomorphic registration. To address this, we introduce a \textit{Joint Affine-Diffeomorphic framework}, based on the large deformations model, that estimates both global affine and local diffeomorphic motions simultaneously within a single optimization. We propose two models: a Full Affine (FA) model that combines affine and diffeomorphic deformations, and a Decomposed Affine (DA) model that restricts the affine part from FA to rotations, translations, and anisotropic scalings. To numerically implement these models for image registration tasks, we develop a tailored optimization strategy that combines progressive affine enrichment, gradually increasing the complexity of the affine component, with a variational weighting scheme that smoothly manages the coarse-to-fine handover between the affine and diffeomorphic components. Evaluated on 2D synthetic datasets and 3D brain MRIs from the IXI cohort, our unsupervised approach avoids the pathological deformations of sequential baselines. We demonstrate that our joint formulation outperforms in Dice overlap two state-of-the-art deep learning foundation models, CARL and uniGradICON, as well as a sequential baseline using FLIRT for the classical affine registration followed by LDDMM. Our implementation is publicly available.\footnote{\url{https://github.com/antonfrancois/Demeter\_metamorphosis/}}
\end{abstract}
\tableofcontents
 \newpage

\section{Introduction \label{sec:intro}}

Shape analysis is a field of study developed in response to the increase of
innovations in medical imaging techniques. Its purpose is to compare several
shapes among a family of shapes to study their variability and perform
statistics on them. In particular, the deformation of a source shape onto a
target shape is a central topic since it makes it possible to track the
different geometric properties of shapes. In this paper we work in the
framework of the diffeomorphic shape spaces introduced by Trouvé et al.
\cite{Trouve1998}; this theory applies to images once the right action of the
diffeomorphism group and the associated discretisations are properly handled.

During MRI acquisition, the patient and the imaging volume are positioned in a
standard orientation; this positioning is only approximate, and orientation,
resolution and quality vary considerably from one image to the other.
Anatomical normalisation is therefore a standard pre-processing step in many
medical imaging analyses.

Comparing images requires spatially aligning them, which is done by finding a
deformation that puts voxels in correspondence and minimises a similarity
measure. The expected deformation must be a diffeomorphism, and we decompose it
into a global affine transformation (e.g.\ rotation, translation, scaling) and a
non-linear component that encodes local differences such as local growth. Most
of the time, the two components are estimated by different techniques: the
standard procedure involves an initial resampling and an affine registration
followed by a non-linear method applied to the affinely aligned image. However,
this two-step approach presents two principal drawbacks: 

First, a dependency on the affine registration output. The affine transformation
is global, whereas anatomical variability is largely local. When an anatomical
structure has undergone a significant deformation, the affine stage can absorb
part of this local deformation and attribute it to global pose or scale. The
residual passed to the non-linear stage is then already biased; in particular,
it may contain translational components that ideally should not influence the
non-linear mapping.

Second, a loss of raw signal fidelity. In radiology, raw MRI signal texture
can carry diagnostically relevant signatures. For instance, texture analysis
performed without prior interpolation has been shown to reveal subtle
microstructural changes related to Alzheimer's pathology
\cite{jytzler2024radiomics}. Each registration step involves an interpolation
of the signal, so a two-step pipeline resamples the image twice rather than
once; reducing the number of interpolations preserves signal fidelity and can
benefit downstream diagnosis or correspondence accuracy.

\paragraph{Related work}
Combining affine and non-linear deformations has been studied for more than two decades. \cite{rueckertNonrigidRegistrationUsing1999} is among the earliest proposals: the global motion is modelled by an affine transformation while the local motion is described by a free-form deformation (FFD) based on B-splines, estimated through a resolution refinement approach.
Later, the FSL suite \cite{jenkinson2001global, jenkinson2002improved} introduced two complementary tools: FLIRT for linear/affine registration, and FNIRT for the subsequent non-linear deformation, the latter being explicitly designed to initialise from the FLIRT affine output. The FSL toolbox remains widely used in the field for its computational efficiency and ease of use.
Among the most popular classical pipelines is Symmetric Normalization (SyN) \cite{avants2008symmetric_ants, avants2009advanced}, distributed within the ANTs toolbox. SyN is a diffeomorphic registration optimising a local normalised cross-correlation, and is symmetric in the sense that the forward and inverse deformations are estimated simultaneously, so that the result does not depend on which image is taken as the source. This symmetry, however, concerns the two directions of the mapping, not the two components of the deformation: in practice, an ANTs registration is run as successive rigid, affine and non-linear SyN stages, each initialised from the transformation returned by the previous one, and the reported transformation is their composition.
The \emph{elastix} toolbox \cite{klein2009elastix} is organised along the same lines: a configurable ITK-based framework supporting rigid, affine and B-spline non-linear registration as a sequence of stages.
All of these methods -- with the notable exception of
\cite{rueckertNonrigidRegistrationUsing1999} -- compute the affine deformation
first, fix it, and only then compute the non-linear part.

Our contribution builds upon the Large Deformation Diffeomorphic Metric Mapping
(LDDMM) framework \cite{Trouve1998, BegMillerTrouveYounes2005}; a detailed
introduction to this framework is provided in \cref{sec:action_images}.
Extensions of this framework by considering models that jointly optimised affine and non-affine deformation were previously explored by \cite{gris2016modular} that combined different types of deformation at the vector field level. More recently, \cite{mouhli2025decoupling} introduced a framework that couples deformations encoded by finite-dimensional Lie group, such as affine transformations with diffeomorphic deformations acting on point clouds via a semidirect product group. Our paper follows this line of research by extending this framework to image registration applications.

Early deep-learning registration methods addressed only the non-linear component of the decomposition introduced above, assuming the inputs had already been affinely aligned by a separate pre-processing step. VoxelMorph \cite{balakrishnan2019voxelmorph} is one of the precursors of this line. It trains a U-Net without supervision, using a similarity and a smoothness term, to predict a displacement vector at every voxel of the image domain -- a dense deformation field -- optionally constrained to be diffeomorphic.
Stronger non-linear models followed: TransMorph \cite{chen2022transmorph} replaces the U-Net with a Swin-Transformer backbone, while LapIRN \cite{mok2020large} relies
on a Laplacian-pyramid coarse-to-fine scheme to reach large deformations. All of
these methods presuppose an affine pre-alignment. A second body of work removes
this assumption by estimating the global transformation jointly with the local
one.

Deep-learning approaches to joint affine and non-linear registration follow two broad strategies. The first regresses the transformation parameters directly from image intensities. It splits further according to whether the global and the local transformations are produced by separate cascaded subnetworks or by a single one. \textbf{Cascaded methods} stack an affine subnetwork with one or more non-linear subnetworks. Each stage warps the moving image before passing it to the next, so that the non-linear component only has to model the residual local deformation. Representative examples are the multi-stage framework of De Vos et al. \cite{de2019deep}, the recursive cascaded networks of Zhao et al. \cite{zhao2019recursive}, and the rigid--affine--non-linear cascade of Strittmatter \cite{strittmatter2023multistage}.
\textbf{Single-network methods} instead emit both the affine parameters and the deformation field from one shared representation. This is typically done through
a coarse-to-fine decoder whose low-resolution levels predict the global transformation and whose high-resolution levels predict the local one, as in
NICE-Trans \cite{meng2023non}. uniGradICON \cite{tian2024unigradicon} follows a
related multi-step approach, but never estimates a separate affine
transformation. It composes several networks, each refining the warp produced by
the previous ones, so that the global alignment is absorbed into this
composition and the global and local parts cannot be read off independently.
Built on the inverse-consistency lineage of GradICON \cite{tian2023gradicon}, it
is offered as a ready-to-use foundation model that generalises across anatomies
and modalities without per-dataset training. As the strongest general-purpose
model of this kind, it is one of the baselines we benchmark against.

The second strategy forgoes direct parameter regression in favour of building geometric structure into the estimator. 
\textbf{Keypoint-based methods} detect corresponding landmarks and solve for the transformation in closed form: an affine by a least-squares fit, and a non-linear warp by thin-plate splines. Both the robustness to large initial misalignments and the interpretability of the result follow from this construction, as in KeyMorph \cite{evan2022keymorph} and its foundation-model extension BrainMorph \cite{wang2024brainmorph}. EasyReg \cite{iglesias2023ready} exploits a similar idea through segmentation: it derives the affine transformation from the centroids of segmentation regions, and estimates the non-linear part with a SynthMorph-style network \cite{hoffmann2024anatomy}.
\textbf{Equivariance-based methods} target the same
robustness through architectural guarantees rather than landmark detection. The principle is that warping an input image should produce a correspondingly warped
output transformation, so that large global misalignments are handled by construction rather than having to be represented in the training data. 
CARL \cite{greer2025carl} achieves this with a coordinate-attention block that computes correspondences as soft centroids of matched image features. This construction remains consistent when the moving and the fixed images are translated independently (termed $[W,U]$ equivariance). The block is then composed with conventional displacement-predicting refinement layers, which supply the fine local deformation.

These last two families are closely related, since aligning equivariant
keypoints is itself an equivariant operation. CARL also illustrates how porous
this taxonomy is: built on the TwoStep composition of the inverse-consistency
lineage \cite{tian2023gradicon}, it is simultaneously a cascade and an
equivariance method, and a sibling of uniGradICON within that lineage.
Common to all of these approaches is that they return a transformation, or a composition of transformations, rather than a trajectory: the affine and the non-linear components are estimated by a single network but do not live in a common variational framework, and no geodesic in shape space is available. 


In our own registration experiments, we observed that an affine stage alone does not always yield a satisfactory global alignment, particularly in the presence of large displacements or differing fields of view. The same difficulty motivates CARL, although the response we propose differs from it.
While deep-learning approaches now dominate the literature, classical
optimisation-based methods remain widely used and competitive. ANTs/SyN \cite{avants2008symmetric_ants} and NiftyReg \cite{langdon2014improving} are still standard tools, and LDDMM captures large deformations while guaranteeing a diffeomorphic transformation \cite{BegMillerTrouveYounes2005,phd_AFrancois}. In the Learn2Reg challenge \cite{hering2022learn2reg}, optimisation-based methods with little or no learning, such as ConvexAdam and corrField, remained among the top performers across several tasks. Yet, as the preceding survey has shown, the joint treatment of the affine and the non-linear components has been explored extensively on the deep-learning side, while it has received comparatively little attention in the classical setting.
To our understanding the trend at the time was to estimate the affine transformation first making them independent problems. In the LDDMM case, a decade ago the family was considered too computationally expensive to be practical, until modern implementations appeared \cite{franccois2021metamorphic, phd_AFrancois, brunn2021fast, hernandez2021combining}.

We aim to build a method within the LDDMM framework, non-statistical and
unsupervised, requiring neither training data nor annotations. Because it is formulated within LDDMM, the resulting transformation is diffeomorphic by construction; more than that, the registration follows a geodesic trajectory in the space of shapes, and this geodesic formulation is applied to the affine component as well as to the non-linear one. Its defining feature is thus to \textbf{produce a continuous registration trajectory -- a time-parameterised deformation rather than a single transformation -- jointly for the global and the local parts}.

\paragraph{Our contributions}

In this paper, we adapt the framework from \cite{mouhli2025decoupling} and consider the simultaneous action of a finite-dimensional Lie group and a group of diffeomorphisms on a shape space, thus extending the classical LDDMM setting. This general theoretical framework encompasses rigid motions and scalings as particular cases.

Our contributions are as follows:
\begin{enumerate}

    \item \textbf{Two joint LDDMM-affine models.}
    We propose two variational, non-statistical models for the simultaneous estimation of a global affine transformation and a local diffeomorphic deformation, formulated within the LDDMM framework. Both produce a time-parameterised geodesic in shape space whose affine and diffeomorphic components can be studied independently.
    We derive the geodesic equations for each model and introduce a change of variables that decouples the two components, simplifying these equations and making the respective contribution of each component analytically transparent.
    The two models share the same diffeomorphic component and differ only in how the affine part is parameterised:
    \begin{itemize}
        \item The \textit{Full Affine} (FA) model pairs the diffeomorphism with an unconstrained affine matrix and a translation vector.
        \item The \textit{Decomposed Affine} (DA) model factorises the affine component into rotation, translation, and anisotropic scaling, each of which can be independently activated or fixed during optimisation.
    \end{itemize}
    

    \item \textbf{Implementation.}
    We provide a full 2D and 3D implementation within the \textsc{Demeter-Metamorphosis} library \cite{franccois2021metamorphic, phd_AFrancois}, a PyTorch-based framework supporting GPU acceleration. This includes two optimisation strategies specific to joint registration: a progressive covector activation scheme and a variational weighting schedule that prevents the diffeomorphic component from dominating the early stages of optimisation.

    \item \textbf{Experimental validation on toy examples.}
    We demonstrate on 2D synthetic examples that joint estimation avoids the local minima that the standard sequential approach --- affine pre-registration followed by independent diffeomorphic optimisation --- is prone to, and that the DA model correctly isolates each geometric component when applied to targets generated by known transformations.

    \item \textbf{Validation against recent deep learning methods.}
    We evaluate our method on a brain template registration task using the IXI dataset, measuring volumetric overlap via Dice scores across anatomical structures. Our classical, optimisation-based method outperforms the deep learning baselines uniGradICON \cite{tian2024unigradicon} and CARL \cite{greer2025carl} without requiring any training data. In particular, we show that joint optimisation produces an affine registration that is a better starting point for diffeomorphic matching than a separately estimated one.

\end{enumerate}

\section{Theoretical framework}\label{Sec:theoretical_framework}

In this section, we present the theoretical framework for the image registration problem by affine and diffeomorphic deformations. Our approach is based on the large deformation framework (LDDMM,\cite{BegMillerTrouveYounes2005}) that we enrich with affine deformations. The resulting group of deformations is constructed as a semi-direct product of the group of diffeomorphisms and the group of affine transformations following the approach described in \cite{mouhli2025decoupling}, and building on ideas developed in \cite{MMRI,BruRiVia,risser2010}. A similar framework has been developed \cite{pierron2024graded} to perform a sequential matching, using a direct product of a finite-dimensional Lie group representing affine motions and a group of diffeomorphisms.
 We also refer to  \cite{gris2015sub} that develop a similar approach by defining a modular deformation framework allowing the combination of several structured vector fields to generate more complex deformations.
 
We begin by recalling several classical results on diffeomorphic transport of images and we then explain how the large deformation framework can be enriched with affine deformations. The main challenges are that the action of diffeomorphisms on images is continuous but not differentiable and the action of affine group on images is ill-defined. Consequently many arguments from the standard orbit model \cite{articleTro1995,ARGUILLERE2015139,pierron2024graded,Pierron2024, PierronThesis} must  be adapted.

\subsection{Diffeomorphic transport of images} \label{sec:action_images}


We start by briefly reviewing the large deformation framework introduced in \cite{Trouve1998,BegMillerTrouveYounes2005}, which is a well-established approach for image registration. The purpose of this setting is to match a source image onto a target using deformations represented as flows of diffeomorphims.
Let $\mathcal{I} := L^2(\Omega,\R)$ be the set of grey scale images on the image domain $\Omega \subset \R^d$, and let $I_S\in\mathcal{I}$ be a source image that we want to match onto a target $I_T\in\mathcal{I}$. We denote $C_0^k(\R^d,\R^d)$ the space of functions that vanish at infinity, and whose derivatives also vanish. Then, we introduce $\operatorname{Diff}_{C_0^k}(\R^d)$ the group of $C_0^k$-diffeomorphisms that tend to identity at infinity, and whose derivatives also tend to identity at infinity, defined by
 \[
\Diff_{C_0^k}(\R^d) = \left(\id+C_0^k(\R^d,\R^d)\right)\cap \Diff^1(\R^d)
\] 
A time-varying vector field $v \in L^2([0,1],C_0^k(\R^d,\R^d))$ can induce a flow of diffeomorphisms belonging to $\operatorname{Diff}_{C_0^k}(\R^d)$ by solving the ODE: $ \dot{\varphi}_t = v_t \circ \varphi_t$ with $\varphi_0=\id$. Therefore $t\mapsto \varphi_t $ is named the flow of $v$.

We recall that in the large deformation framework, the trajectory of an image is defined as the action of a flow of diffeomorphisms in $\Diff_{C_0^k}(\Omega)$ on the template image \cite{Dupuis1998,Younes2019} : \[I_t  = \varphi_t\cdot I_S := I_S \circ \varphi_t^{-1}.\] This action is only continuous because of the composition on the right by the inverse of $\varphi$. This means, in particular, that more regularity on the image $I$ is required to differentiate this action and to define the infinitesimal action $\xi_I : v \mapsto v \cdot I := \partial_{\varphi}(\varphi \cdot I) \vert_{\varphi=\id} v$ \cite{ARGUILLERE2015139}. Indeed, as stated in the next proposition, for $I\in H^1(\Omega,\R)$, the differential of the action $\varphi \mapsto \varphi \cdot I$ exists and is well-defined:
\[
\xi_I:\app{C_0^k(\Omega,\R^d)}{\mathcal{I}}{v}{-\langle \nabla I,v\rangle}
\]
The evolution of the deformed image can now be derived from this infinitesimal action 
\begin{equation*}
\label{dyn_I}
    \dot{I}_t = v_t \cdot I_t := -\langle \nabla I_t , v_t \rangle, \quad I_0=I_S
\end{equation*} where $v_t=\dot{\varphi}_t\circ\varphi_t^{-1}$ is the Eulerian derivative. 

\begin{proposition}[Deformation of images]
\label{flow_images}
    Suppose $I_S\in H^1(\Omega,\R)$ and $v\in L^2([0,1],C_0^1(\Omega,\R^d))$. Then there exists a unique global solution of 
    \begin{equation}
    \label{Ani:eq_image}
            \dot{I}_t=-\langle\nabla I_t, v_t\rangle, \quad I_0=I_S.
    \end{equation}
    Moreover this unique solution $I_t$ also satisfies the integrated equation
    \begin{equation}
    \label{Ani:eq_image_int}
    I_t = I_S\circ\varphi_t^{-1}
    \end{equation}
    where $\varphi_t$ is the flow of $v_t$.
\end{proposition}
\begin{remark}
    The integrated version gives us in particular that the flow $I_t$ belong to $H^1(\Omega,\R)$ for all time $t$, but is only absolutely continuous when $I_t$ lies in $L^2$.
\end{remark}
\begin{proof}
Let $v\in L^2([0,1],C_0^1(\Omega,\R^d))$. We consider the curve $\varphi\in AC_{L^2}([0,1],\operatorname{Diff}_{C_0^{1}}(\R^d))$ which is the flow of $v$, that is the unique solution of the ODE
\[
\dot{\varphi}_t = v_t\circ\varphi_t, \quad \varphi_0=\id.
\]
    Since the restriction of the action to $H^1$ images
    \[
    \app{\operatorname{Diff}_{C_0^1}(\R^d)\times H^1(\Omega,\R)}{L^2(\Omega,\R)}{\varphi, I}{I\circ\varphi^{-1}}
    \] is $C^1$, the curve $t\mapsto I_S\circ\varphi_t^{-1}\in H^1(\Omega,\R)$ is therefore absolutely continuous in $\mathcal{I}=L^2(\Omega,\R)$ and is solution of \cref{Ani:eq_image}. 
    
    Conversely, suppose $t\mapsto I_t$ is an absolutely continuous curve in $\mathcal{I}$, such that for any $t\in I$, $I_t\in H^1(\Omega,\R)$ and $I_t$ is solution of \cref{Ani:eq_image}. We prove now that for all $t\in [0,1]$,
    \[
    I_t\circ\varphi_t=I_S.
    \]
    The curve $t\mapsto I_t\circ\varphi_t$ is not necessarily absolutely continuous in $\mathcal{I}$, since we only have $I_t\in\operatorname{AC}_{L^2}([0,1],\mathcal{I})$. However for any $h\in C_0^\infty(\Omega,\R)$, we can derivate $t\mapsto \int_{\Omega}\langle I_t\circ\varphi_t(x),h(x) \rangle dx$ and we get
\begin{align*}
    \frac{d}{dt}\int_{\Omega} I_t(\varphi_t(x))\,h(x) dx &= \int_\Omega \left[ \dot{I}_t(\varphi_t(x)) + \langle\nabla I_t(\varphi_t(x)), \dot{\varphi}_t(x) \rangle \right] h(x) dx \\
    &= \int_\Omega \left[ -\langle\nabla I_t(\varphi_t(x)), v_t(\varphi_t(x)) \rangle + \langle \nabla I_t(\varphi_t(x)), v_t(\varphi_t(x)) \rangle \right] h(x) dx \\
    &=0
   \end{align*}
from which the result follows.
\end{proof}

Following, the large deformation framework, we consider vector fields belonging to a Reproducing Kernel Hilbert Space (RKHS) of vector fields $V$, continuously embedded in $C_0^1(\R^d,\R^d)$. The RKHS structure is particularly interesting because it relies only on a kernel function to generate the entire space, leading to a metric that is both mathematically tractable and computationally straightforward to implement.
We formulate the standard registration problem, that consists in matching the source $I_S$ onto the target $I_T$ using diffeomorphic deformations and expressed as the following minimization problem:

\begin{eqnarray}
    \label{eq:lddmm_image}
\inf_{v\in L^2({[0,1],V)}} J(v) 
&=& \int_0^1   \frac{1}{2}   \Vert v_t \Vert_V^2 \, dt 
+ \mathcal{D}(I_1) \\
     \text{ s.t }& &     \left\{
        \begin{array}{ll}
\dot{I}_t = -\langle\nabla I_t, v_t\rangle, \\
 I_0=I_S        
        \end{array} 
        \right.\notag
\end{eqnarray}
where $\mathcal{D} : \mathcal{I}\to\R$ is a data attachment term that measures the discrepancy between the target $I_T$ and the deformed image $I_1$. If $\mathcal{D}$ is continuous, then this problem admits minimizers \cite{Pierron2024,ARGUILLERE2015139}.

The Pontryagin Maximum Principle (PMP) \cite{ARGUILLERE2015139,pierron2024graded} corresponding to the problem \eqref{eq:lddmm_image} provides necessary conditions for optimal solutions. Consider the following Hamiltonian:\cite{ARGUILLERE2015139,BegMillerTrouveYounes2005}
\begin{equation}\label{eq:hamilt}
H(I,p,v) = (p \mid - \langle \nabla I,v\rangle) - \frac{1}{2} \Vert v \rVert_V^2    
\end{equation}

where $K_V$ is the Riesz isomorphism corresponding to the reproducing kernel associated with $V$ and $p_t\in L^2(\Omega,\R^d)$ is a covector. Therefore, $v \in L^2([0,1],V)$ is an optimal solution of \eqref{eq:lddmm_image} if there exists a time-dependent covector $p_t$ such that the following Hamiltonian equations are satisfied:
\begin{equation} \label{eq:standard_hamilt_eq}
    \left\{
    \begin{aligned}
        \dot I_t &= \partial_p H(I_t,p_t,v_t) \\
        \dot p_t &= - \partial_I H(I_t,p_t,v_t) \\
        0 &= \partial_v H(I_t,p_t,v_t)
    \end{aligned}
    \right.
    \quad \Longleftrightarrow \quad
    \left\{
    \begin{aligned}
        \dot I_t &= - \langle \nabla I_t, v_t \rangle \\
        \dot p_t &= - \operatorname{div}(p_t v_t) \\
        v_t &= - K_V(p_t\nabla I_t)
    \end{aligned}
    \right.
\end{equation}
with the endpoint condition $p_1=-\partial_I\mathcal{D}(I_1)$.
Note that the Hamiltonian equation for $I$ and $p$ take the form of an advection equation and a continuity equation, respectively. By replacing the optimal expression of $v$, the Hamiltonian simplifies to
\[
H(I,p) = \frac{1}{2} \lVert K_V(p\nabla I)\rVert_V^2
\]

Therefore, the variational problem \eqref{eq:lddmm_image} can be expressed as a minimization problem over the covector $p$.

\begin{eqnarray}
    \label{eq:lddmm_image_p}
\inf_{p} J(p) 
&=& \int_0^1   \frac{1}{2}   \Vert K_V(p_t\nabla I_t)  \Vert_V^2 \, dt 
+ \mathcal{D}(I_1) \\
    & & \text{ s.t \cref{eq:standard_hamilt_eq} are satisfied and } I_0=I_S     \notag
\end{eqnarray}
Finally, by conservation property, is it sufficient to minimize with respect to the covector at initial time $p_0$, which is called geodesic shooting.

The large deformation framework presented in this section serves as the foundation for the models developed in this work. By enriching the space of diffeomorphisms with affine motions, we will derive new geodesic equations distinct from the standard ones given in \eqref{eq:standard_hamilt_eq}, leading to new minimization problems adapted to the joint affine and diffeomorphic deformation framework.

\subsection{The group of affine motions} \label{sec:affine_motions}

Building upon the large deformation introduced previously, this section extends the model to incorporate affine motions in addition to standard diffeomorphic deformations. We refer to this framework as the \textit{Full Affine} model.

\subsubsection{First definitions and properties} \label{sec:affine_motions:action}

The Lie group of affine motions is defined by the semi-direct product

\begin{equation}
    \operatorname{Aff}(\mathbb{R}^d)  \coloneq \operatorname{GL}_d(\mathbb{R}) \ltimes \mathbb{R}^d
\end{equation}
with group laws given by
\[
\left\{\begin{array}{l}
(A,b)(A',b') = \left(AA',Ab'+b\right)  \\
      (A,b)^{-1}=(A^{-1},-A^{-1}b).
\end{array} \right.
\]
Its corresponding Lie algebra is denoted \begin{equation}
    \mathfrak{aff}(\R^d) := \R^{d \times d} \oplus \R^d
\end{equation}

The group $\Aff(\R^d)$ naturally acts on $\R^d$ via $(A,b) \cdot x = Ax +b$. We therefore define the action of $\Aff(\R^d)$ on the space of images $\mathcal{I}$ by:

\begin{equation}
\label{rigid_action_I1}
    ((A,b)\cdot I)(x)= I(A^{-1}(x-b)) \, \text{ for } x \in \Omega
\end{equation}

However this does not exactly define an action since  images are defined on the open set $\Omega \subset \R^d$ that is not necessarily stable by the action of the affine group. This means that a loss of information may occur when the affine motion is too large. However, in most of applications, the support of images is strictly contained within $\Omega$ and we can affinely transport them without loss of information.

\begin{proposition}[Transport of the support]\label{prop:transport_support}
Let $I\in\mathcal I$ and denote by $\Supp(I)\subset\Omega$ its support.
Let $(A,b),(A',b')\in\Aff(\mathbb R^d)$ and define
\[
S_1 \coloneq A'\Supp(I) + b', \qquad
S_2 \coloneq A S_1 + b \;=\; A\bigl(A'\Supp(I)+b'\bigr) + b .
\]
Assume that $S_1,S_2 \subset \Omega$, then
\[
(A,b)\cdot\bigl( (A',b')\cdot I \bigr)
= (AA',\, Ab' + b)\cdot I.
\]
Moreover $S_1=\Supp\bigl((A',b')\cdot I\bigr)$ and
$S_2=\Supp\bigl((A,b)\cdot((A',b')\cdot I)\bigr)$.
\end{proposition}

\begin{proof}
Recall that for any $(A,b)\in\Aff(\mathbb R^d)$ the transformed support satisfies
\[
\Supp\bigl((A,b)\cdot I\bigr) = A\Supp(I)+b,
\]
since $((A,b)\cdot I)(x)=I(A^{-1}(x-b))$ is nonzero if and only if
$A^{-1}(x-b)\in\Supp(I)$, i.e.\ $x\in A\Supp(I)+b$.
Thus $S_1=\Supp((A',b')\cdot I)$, and the assumption $S_1\subset\Omega$
guarantees no loss of information after the first action.

Now for $x\in\Omega$,
\[
\begin{aligned}
\bigl((A,b)\cdot((A',b')\cdot I)\bigr)(x)
&= ((A',b')\cdot I)\bigl(A^{-1}(x-b)\bigr) \\
&= I\!\left(A'^{-1}\!\left(A^{-1}(x-b)-b'\right)\right).
\end{aligned}
\]
We rewrite the argument as
\[
A'^{-1}\bigl(A^{-1}(x-b)-b'\bigr)
= (AA')^{-1}\bigl(x - (b + Ab')\bigr).
\]
Hence
\[
\bigl((A,b)\cdot((A',b')\cdot I)\bigr)(x)
= I\!\left((AA')^{-1}(x-(b+Ab'))\right)
= \bigl((AA',\,Ab'+b)\cdot I\bigr)(x),
\]
which proves the desired equality of actions.

Finally, applying the support formula a second time yields
\[
\Supp\bigl((A,b)\cdot((A',b')\cdot I)\bigr)
= A\,\Supp\bigl((A',b')\cdot I\bigr) + b
= A S_1 + b = S_2,
\]
and the assumption $S_2\subset\Omega$ prevents any loss of information.
\end{proof}

To define more complex deformations, we enrich the group of diffeomorphisms with the group of affine motions by the introduction of the following semidirect product:
\begin{equation}\Aff(\mathbb{R}^d)\ltimes \Diff_{C_0^k}(\mathbb{R}^d).
\end{equation}
The associated group laws  given by
\[
\left\{\begin{array}{l}
    (A,b,\varphi)(A',b',\varphi') = (AA',Ab'+b,A'^{-1}\varphi(A' \varphi'+b') - A'^{-1}b')  \\
      (A,b,\varphi)^{-1}=(A^{-1},-A^{-1}b,A\varphi^{-1}(A^{-1}(\id-b))+b).
\end{array} \right.
\] 
and its action on the space of images $\mathcal{I}$ is defined by
\begin{equation*}
    ((A,b,\varphi)\cdot I)(x) = I(\varphi^{-1}(A^{-1}(x-b)).
\end{equation*}

Note that the stability issue of $\Omega$ under the action of the affine group arises here as well. However, it is similarly resolved using \cref{prop:transport_support}.

As detailed in \cref{sec:action_images}, for $I \in \mathcal{I}=L^2(\Omega, \R)$, the action is not regular enough to define an infinitesimal action. Therefore, we will assume that $I \in H^1(\Omega,\R^d)$.

\begin{proposition}
    Let $I_S \in H^1(\Omega,\R ^d)$ be a template image. For $(M,\beta,v) \in L^2([0,1],\aff(\R^d) \oplus C_0^k(\R^d,\R^d))$, the evolution equation of $I_S$ is
        \begin{equation}\label{eq:dyn_I}
            \dot{I}_t = - \langle \nabla I_t(x),v_t(x) + M_tx + \beta_t\rangle
        \end{equation}
where  $I_0=I_S$.
\end{proposition}

\begin{proof}
The time-dependent vector field $(M,\beta,v) \in L^2([0,1],\mathfrak{aff}(\R^d) \oplus C_0^k(\R^d ,\R^d))$ defines an unique flow $(A_t,b_t,\varphi_t)$ in $\Aff(\R^d) \ltimes \Diff_{C_0^k}(\R^d)$ via the resolution of following ODE:  
\begin{align*} \label{evolution_equation_affine}
(\dot{A}_t,\dot{b}_t,\dot{\varphi}_t) &= (M_t,\beta_t,v_t) (A_t,b_t,\varphi_t) \notag \\
&= \left( M_t A_t, M_t b_t + \beta_t, A_t^{-1} v_t(A_t\varphi_t + b_t) \right),
\end{align*}
with the initial conditions $(A_0,b_0,\varphi_0)=(I_d,0,\id)$.

Recall that the deformation of a source image $I_S$ by the action of $\Aff(\R^d) \ltimes \Diff_{C_0^k}(\R^d)$ defines the trajectory of images
\begin{equation*} \label{affine_action}
    I_t(x) :=((A_t,b_t,\varphi_t)\cdot I_S)(x) = I_S(\varphi^{-1}_t (A_t^{-1}(x-b_t)))
\end{equation*}
for $x \in \Omega$.
Therefore, differentiating this action with respect to $(A_t,b_t,\varphi_t)$ and evaluating at time $t=0$ yields the infinitesimal action 
\begin{equation*}
    \xi_I : \app{\mathfrak{aff}(\R^d) \oplus C_0^k(\Omega,\R^d)}{\mathcal{I}}{(M,\beta,v)}{x \mapsto -\langle \nabla I(x),v(x) + Mx + \beta \rangle}
\end{equation*}
The evolution equation of the deformed image $I_t$ directly comes from $\dot{I}_t = \xi_{I_t}(M_t,\beta_t,,v_t)$.
\end{proof}

\subsubsection{Matching problem and Hamiltonian equations}\label{sec:aff_match_problem}

Following the approach presented in \cite{mouhli2025decoupling}, we define an augmented space of images $\tilde{\mathcal{I}} = \Aff(\R^d) \times \mathcal{I}$ by enriching the standard space of images $\mathcal{I}$ with additional geometric information, such as the orientation and scale of the image. In particular, this formulation will allow us to later introduce a change of variables that factors out the affine motion from the deformed image, isolating the purely diffeomorphic component of the deformation.

We consider the action of $\Aff(\R^d) \ltimes \Diff_{C_0^k}(\R^d)$ on $\tilde{\mathcal{I}}$ defined by
\begin{equation} \label{eq:augmented_action}
    (A,b,\varphi)\cdot(\tilde{A},\tilde{b},I) = (A\tilde{A}, A \tilde{b} + b, I\circ\varphi^{-1} \circ A^{-1} \circ (\id - b)).
\end{equation}
By differentiating this action, we deduce the corresponding infinitesimal action
\[
\xi_{(\tilde{A},\tilde{b},I)}(M,\beta,v) = 
\big( M \tilde{A},\, M \tilde{b} + \beta,\, - \langle \nabla I, v + M\cdot + \beta \rangle \big),
\]
and the evolution equation of the deformed shape $(\tilde{A}_{t},\tilde{b}_{t},I_t)$
\begin{equation}
\label{augmented_flow}
\begin{cases}
    \dot{\tilde{A}}_{t} = M_t \tilde{A}_{t}, & \tilde{A}_{0} = I_d \\
    \dot{\tilde{b}}_{t} = M_t \tilde{b}_{t} + \beta_t, & \tilde{b}_{0} = 0 \\
    \dot{I}_t = - \langle \nabla I_t, v_t + M_t x + \beta_t \rangle, & I_0 = I_S
\end{cases}
\end{equation}

We define a new image registration problem as the following variational problem:

\begin{equation}
\label{eq:augmented_matching_aff1}
\inf_{(M,\beta,v)} 
J(M,\beta,v) =
\frac{1}{2} \int_0^1 \big( \| M_t \|^2 + \| \beta_t \|^2 + \| v_t \|_V^2 \big) dt
+ \mathcal{D}^{FA}_{\gamma}\big( \tilde{A}_{1}, \tilde{b}_{1}, I_1 \big),
\end{equation}
subject to the dynamics given by \cref{augmented_flow}. 
We define the data attachment term as the sum of two components
\begin{equation}
\mathcal{D}^{FA}_\gamma(\tilde{A}_1,\tilde{b}_1,I_1)=\gamma \Vert (\tilde{A}_1,\tilde{b}_1)\cdot I_S - I_T\Vert_{L^2}^2 + (1 -\gamma) \Vert I_1  - I_T\Vert_{L^2}^2 
\label{eq:attachment_term_affine}
\end{equation}
with $0\leq\gamma \leq 1$. The first component ensures that the affine transformation roughly aligns the source with the target. The second component guarantees that the total deformation of the template, induced by both the affine and diffeomorphic deformations, effectively matches the target. An additional advantage of using these two components, rather than the standard data attachment term $\mathcal{D}(I_1)=\Vert I_1 - I_T\Vert^2_{L^2}$, is that it relax the ending condition. 

Problem \eqref{eq:augmented_matching_aff1} can be interpreted as an optimal control problem, where the vector fields $(M,\beta,v)$ act as control variables as presented in \cref{sec:action_images}. By applying the PMP, we can derive geodesic equations through a Hamiltonian formulation of this problem. The geodesics correspond exactly to the critical points of the minimization functional. The Hamiltonian associated with the variational problem \eqref{eq:augmented_matching_aff1} is
    \begin{eqnarray}
H(\tilde{A},\tilde{b},I,p^A,p^b,p^I,M,\beta,v) &=&  (p^A \vert M\tilde{A}) + (p^b \vert M\tilde{b} + \beta) + \left(p^I\mid \langle-\nabla I,v + M\cdot + \beta\rangle\right)\nonumber\\
&-& \frac{1}{2}\left(  \Vert M\Vert^2 + \Vert \beta \Vert^2+\vert v\Vert_V^2 \right)      
\label{eq:hamil_affine}
\end{eqnarray}
where $(p^A,p^b,p^I)$ are covectors belonging to the cotangent bundle of the augmented space of images  $T^*(\Aff(\R^d) \times I)$
\begin{proposition}\label{prop:geod_eq_aff}
    The critical points of problem \eqref{eq:augmented_matching_aff1} satisfy the Hamiltonian equations associated with the Hamiltonian \eqref{eq:hamil_affine}:
\begin{equation}
\label{eq:Hamiltonian_full_system0}
\left\{
    \begin{aligned}
        &\dot{\tilde{A}}_t = M_t\tilde{A}_t \\
        &\dot{\tilde{b}}_t = M_t \tilde{b}_t + \beta_t \\
        &\dot{{I}}_t(x) = - \langle \nabla I_t(x), v_t (x) + M_tx + \beta_t \rangle \\
        &\dot{p}^A_t = -  M^{\top}_tp^A_t \\
        &\dot{p}^b_t = - M^{\top}_t p_t^b \\
        &\dot{p}^I_t = - \operatorname{div}(p^I_t(v_t + M_t x +\beta_t))
    \end{aligned}
\right. \quad \text{ where }
\quad
\left\{
    \begin{aligned}
        &M_t = - \int_{\Omega} p^I_t(x) \nabla I_t(x) x^\top dx + p^A_t \tilde{A}_t^{\top} + p^b_t \tilde{b}_t^{\top} \\
        &\beta_t = - \int_{\Omega} p^I_t(x) \nabla I_t(x) \,dx + p^b_t \\
        &v_t = -K_V(p^I_t \nabla I_t) \notag
    \end{aligned}
\right.
\end{equation}
\end{proposition}

\begin{proof}
    Let us consider the Hamiltonian $H$ defined in \cref{eq:hamil_affine}.
The PMP  states that the critical points satisfy the following Hamiltonian equations:
\begin{equation}
\left\{
    \begin{aligned}
        &(\dot{\tilde{A}}_{t},\dot{\tilde{b}}_{t},\dot{I}_t,\dot{p}^A_t,\dot{p}^b_t,\dot{p}^I_t)=(\partial_{p^A} H,\partial_{p^b} H,\partial_{p^I} H,- \partial_{\tilde{A}} H,- \partial_{\tilde{b}} H,- \partial_I H) \\
        &(\partial_M H , \partial_\beta H,\partial_v H)=0 
    \end{aligned}
\right.
\end{equation}
First, let us explicitly expand the terms composing the first part of the Hamiltonian: 
\begin{align*}
    &(p^A \vert M\tilde{A}) = \operatorname{tr}((p^A)^\top M\tilde{A}) \\
    &(p^b \vert M\tilde{b} + \beta)=\langle p^b,M\tilde{b} + \beta \rangle \\
    &\left(p^I\mid \langle-\nabla I,v + M\cdot + \beta\rangle\right) =- \int_{\Omega} p^I(x)\langle\nabla I(x),v(x) + Mx + \beta\rangle dx.
\end{align*}

Consequently, taking the partial derivatives of the Hamiltonian with respect to the control variables $(M,\beta,v)$ yields:
\begin{equation}
\left\{
    \begin{aligned}
        & \partial_MH = 0 \Longleftrightarrow M=p^A\tilde{A}^\top+p^b\tilde{b}^\top -\int_{\Omega}p^I(x) \nabla I(x)x^{\top}dx , \\
        & \partial_{\beta}H = 0 \Longleftrightarrow \beta= p^b - \int_{\Omega}p^I(x)\nabla I(x)dx \\
        & \partial_{v}H =0 \Longleftrightarrow v = -K_V(p^I\nabla I)
    \end{aligned}
\right.
\end{equation}
which directly provides the equations for the optimal controls. 

The differential equations of $p^A$ and $p^b$ follow directly from the partial derivatives of $H$. We now focus on the equation for $p^I$ which requires more precisions. Using the product rule for the divergence operator, it follows that $\langle \nabla I , (v + M\cdot + \beta)p^I \rangle = \operatorname{div}((v + M\cdot + \beta)p^I I) - I \operatorname{div} ((v + M\cdot + \beta)p^I)$. Then, the Hamiltonian equation corresponding to $p^I$ gives:

\begin{align*}
    \dot{p}^I_t &= -\partial_{I} H \\
    &= \partial_I \int_{\Omega} p^I(x)\langle\nabla I(x),v(x) + Mx + \beta\rangle dx  \\
    &= \partial_I \int_\Omega \operatorname{div}((v + M\cdot + \beta)p^I I) (x) dx - \partial_I \int_\Omega  I(x)\operatorname{div} ((v + M\cdot + \beta)p^I)(x) dx
\end{align*}

Since the image vanishes outside its support $\operatorname{Supp}(I) \subset \Omega$, the divergence theorem implies that $\int_\Omega \operatorname{div}((v + M\cdot + \beta)p^I I)(x) dx =0$. Consequently, computing the derivative with respect to $I$ via the $L^2$ inner product identification yields:
\begin{align*}
    \dot{p}^I_t &=  - \partial_I \int_\Omega  I(x) \operatorname{div} ((v + M\cdot + \beta)p^I)(x) dx \\
    &= - \partial_I \big(I \mid\operatorname{div}((v+M\cdot+\beta)p^I)\big) \\
    &= -\operatorname{div}((v+M\cdot+\beta)p^I) 
\end{align*}
\end{proof}

The previous Hamiltonian can be simplified by replacing the controls $(M,\beta,v)$ by their optimal expression:
\begin{eqnarray}
H(\tilde{A},\tilde{b},I,p^A,p^b,p^I) &=& \frac{1}{2}  \left\Vert -\int_{\Omega} p^I(x) \nabla I(x) x^\top dx  + p^A \tilde{A}^{\top} + p^b \tilde{b}^{\top} \right\Vert^2  \\
    &+&\frac{1}{2} \left\Vert -\int_{\Omega} p^I(x) \nabla I(x) \,dx  + p^b\right\Vert^2 
    + \frac{1}{2} \left\Vert K_V(p^I \nabla I) \right\Vert^2       \notag
\end{eqnarray}
Consequently, the minimization over control variables $(M,\beta,v)$ in \cref{eq:augmented_matching} can be rewritten as a minimization over covectors.
\begin{eqnarray}
    \label{eq:lddmm_image2}
\inf_{(p^A,p^\beta,p^I)} 
J(p^A,p^b,p^I)
&=& \frac{1}{2} \int_0^1 \big( \| M_t \|^2 + \| \beta_t \|^2 + \| v_t \|_V^2 \big) dt
+ \mathcal{D}^{FA}_{\gamma}\big( \tilde{A}_{1}, \tilde{b}_{1}, I_1 \big), \\
     \text{ s.t }& &  \left\{
        \begin{array}{ll}
\text{Eq. from \cref{prop:geod_eq_aff}} \\
 I_0=I_S        
        \end{array} 
        \right.\notag
     \notag
\end{eqnarray}

Following \cite{mouhli2025decoupling}, we proceed to a change of variable to define a new image:
\begin{equation}\label{eq:change_variable}
    \tilde{I}(x):=(\tilde{A},\tilde{b})^{-1} \cdot I (x)= I(\tilde{A}x+\tilde{b})
\end{equation}
This change of variable allows to separate the contribution of affine and diffeomorphic motions in the global deformation since it allows us to represent the image in its intrinsic frame, factoring out the action of the affine motions. Indeed, the dynamic of the new image is
\begin{equation}
    \dot{\tilde{I}}_t 
    = - \langle \nabla\tilde I_t(x),\tilde{A}_t^{-1}v(\tilde{A}_tx +\tilde{b}_t)\rangle \\
\end{equation}
which only involves the diffeomorphic vector field, in contrast to the dynamic of $I$ previously established in \cref{eq:dyn_I}. Therefore, we define a new variational problem we are interested in is

\begin{eqnarray}
\label{eq:augmented_matching}
\inf_{(M,\beta,v)} 
J(M,\beta,v) &=&
\frac{1}{2} \int_0^1 \big( \| M_t \|^2 + \| \beta_t \|^2 + \| v_t \|_V^2 \big) dt
+ \tilde{\mathcal{D}}^{FA}_{\gamma}\big( \tilde{A}_{1}, \tilde{b}_{1}, \tilde{I}_1 \big), \\
     \text{ s.t }& &  \left\{
        \begin{array}{ll}
        \dot{\tilde{A}}_t = M_t\tilde{A}_t \\
        \dot{\tilde{b}}_t = M_t \tilde{b}_t + \beta_t \\
        \dot{\tilde{I}}_t(x) = - \langle \nabla \tilde{I}_t(x), \tilde{A}_t^{-1}v_t(\tilde{A}_t x+\tilde{b}_t) \rangle, \quad 
 I_0=I_S        
        \end{array} 
        \right.\notag
     \notag
\end{eqnarray}
with the data attachment term  $\tilde{D}^{FA}_\gamma(\tilde{A}_1,\tilde{b}_1,\tilde{I}_1)= D_{\gamma}^{FA}(\tilde{A}_1,\tilde{b}_1,(\tilde{A}_1,\tilde{b}_1)\cdot\tilde{I}_1).$
The Hamiltonian associated with this new variational problem becomes
\begin{equation}
\label{eq:pre_hamil_decoupled}
\begin{aligned}
H(\tilde{A},\tilde{b},\tilde{I},p^A,p^b,\tilde{p}^I,M,\beta,v) &= ( p^A \mid M \tilde{A} ) + ( p^b \mid M \tilde{b} + \beta) + \big( \tilde{p}^I \mid \langle -  \nabla \tilde{I}, \tilde{A}^{-1}v(\tilde{A} \cdot+\tilde{b}) \rangle \big) \\
    &\quad - \frac{1}{2} \big( \Vert M \Vert^2 + \Vert \beta \Vert^2 + \Vert v \Vert_V^2 \big),
\end{aligned}
\end{equation}

Similarly to the transformation applied to images in \cref{eq:change_variable}, we introduce a change of variable for the vector fields:
\begin{equation}
    \label{eq:vtilde}
    \tilde{v}_t(x) = \tilde{A}_t^{-1}v_t(\tilde{A}_tx+\tilde{b}_t).
\end{equation}
Then, the Hamiltonian equations can be simplified as established in the next proposition.

\begin{proposition}\label{prop:geod_eq_decoupled}
    The  critical points associated with the variational problem \eqref{eq:augmented_matching}, expressed with the vector field $\tilde{v}_t$, satisfy the Hamiltonian equations
\begin{equation}
\label{eq:Hamiltonian_full_system_decoupled_symp}
\left\{
    \begin{aligned}
        &\dot{\tilde{A}}_t = M_t\tilde{A}_t \\
        &\dot{\tilde{b}}_t = M_t \tilde{b}_t + \beta_t \\
        &\dot{\tilde{I}}_t = - \langle \nabla \tilde{I}_t, \tilde{v}_t \rangle \\
        &\dot{p}^A_t = - M_t^\top p_t^A - \tilde{A}_t^{-\top} \left( \int_{\Omega}  \ \tilde{p}_t^I(x) \nabla \tilde{I}_t(x) \tilde{v}_t(x)^\top dx -  \int_{\Omega} d\tilde{v}_t(x)^\top \tilde{p}_t^I(x)\nabla \tilde{I}_t(x) x^\top dx \right)\\
        &\dot{p}^b_t = - M_t^\top p_t^b +  \tilde{A}_t^{-\top}\int_\Omega  d\tilde{v}_t(x)^\top \tilde{p}_t^I(x) \nabla \tilde{I}_t(x)  dx  \\
        &\dot{\tilde{p}}^I_t = - \operatorname{div}(\tilde{p}_t^I\tilde{v}_t)  
    \end{aligned}
\right.
\end{equation}
where the optimal controls are given by
\begin{equation}
\label{eq:optimal_controls_decoupled}
\left\{
    \begin{aligned}
        &M_t = p_t^A \tilde{A}_t^\top + p_t^b \tilde{b}_t^\top \\
        &\beta_t =  p_t^b \\
        &\tilde{v}_t = -K_{\tilde{V} }(\tilde{p}_t^I \nabla \tilde{I}_t )
    \end{aligned}
\right.
\end{equation}
with $K_{\tilde{V}}$ the Riesz isomorphism generated by the kernel $ k_{\tilde{V}}(x,y)=A^{-1} k_V(Ax+b,Ay+b)A^{-\top}$
\end{proposition}

 \begin{proof}
    The proof is similar to the one given in \cref{prop:geod_eq_aff}
\end{proof}

This change of variables yields new geodesic equations compared to the ones given in \cref{prop:geod_eq_aff}. First, in the original setting, the image dynamic $\dot{I}_t$ was driven by the vector field $x \mapsto v_t(x) + M_t x + \beta_t$ that combines both affine and unstructured diffeomorphic deformations. After the change of variables, the new image $\tilde{I}_t$ evolves exclusively under the action of the new vector field $\tilde{v}_t(x) = \tilde{A}_t^{-1}v_t(\tilde{A}_tx+\tilde{b}_t)$, effectively removing the explicit dependence on the affine controls from the image evolution equation. On the other hand, the optimal controls associated with the affine part, namely $M$ and $\beta$, do not depend directly on the image $\tilde{I}$, but only on their associated covectors $p^A$ and $p^b$. 
Consequently, the computational complexity shifts from the controls to the evolution equations of these affine covectors. While the differential equation for $p^A_t$ might appear computationally costly due to the spatial integration over $\Omega$ and the presence of the Jacobian $d\tilde{v}_t$, it is efficient in practice. The spatial differential $d\tilde{v}_t(x)$ can be easily evaluated using finite differences, yielding a $d \times d$ matrix for each of the $N$ pixels. The integral thus reduces to straightforward pixel-wise multiplications of $d \times d$ matrices, followed by a simple global sum over the image domain. 
Note that if $k_V$ is a translation-invariant kernel (such as a Gaussian), the expression of the transformed kernel $k_{\tilde{V}}$ naturally simplifies, since $k_V(\tilde{A}x+\tilde{b}, \tilde{A}y+\tilde{b}) = k_V(\tilde{A}x, \tilde{A}y)$.

By replacing the controls $(M,\beta,v)$ with their optimal expressions the Hamiltonian \eqref{eq:pre_hamil_decoupled} simplifies to
\begin{equation}
\begin{aligned}
    H(\tilde{A},\tilde{b},\tilde{I},p^A,p^b,\tilde{p}^I) &= \frac{1}{2} \Vert p^A\tilde{A}^\top + p^b \tilde{b}^{\top} \Vert^2 + \frac{1}{2}\Vert p^b \Vert^2 + \frac{1}{2} \Vert K_{\tilde{V}} (\tilde{p}_t^I \nabla \tilde{I}_t )\Vert_{\tilde{V}}^2
\end{aligned}
\end{equation}
Consequently, the minimization over the control variables can be replaced by a minimization over the initial covectors to perform geodesic shooting:
\begin{eqnarray}
    \label{eq:lddmm_image_decoupled}
\inf_{(p^A_0,p^\beta_0,\tilde{p}^I_0)} J(p^A_0,p^b_0,\tilde{p}^I_0) &=& \int_0^1 H\big(\tilde{A}_t,\tilde{b}_t,\tilde{I}_t,p^A_t,p^b_t,\tilde{p}^I_t\big) dt + \tilde{\mathcal{D}}_{\gamma}\big( \tilde{A}_{1}, \tilde{b}_{1}, \tilde{I}_1 \big), \\
     \text{ s.t }& &  \left\{
        \begin{array}{ll}
\text{Eq. from \cref{prop:geod_eq_decoupled}} \\
 \tilde{I}_0=I_S, \quad \tilde{A}_0 = I_d, \quad \tilde{b}_0 = 0      
        \end{array} 
        \right.\notag
\end{eqnarray}

\subsection{Anisotropic scaling and isometries}\label{sec:decomp_affine}

This section is dedicated to the particular case where we restrict affine motions  to translations, rotations and anisotropic scaling . We refer to this framework as the \textit{Decomposed Affine} (DA) model, in opposition with the \textit{Full Affine} presented in \cref{sec:affine_motions}.

\subsubsection{First definitions and properties}

In \cref{sec:affine_motions}, we  considered the group of affine motions $\Aff(\R^d) := \operatorname{GL}_d \ltimes \R^d$ acting on images via $(A,b) \cdot I (x) = I(A^{-1}(x-b))$. Depending on the application, it is often more relevant to decompose the affine motion into distinct, interpretable sub-components such as rotations, translations, or anisotropic scalings. Indeed, by polar decomposition, an invertible matrix $A \in \operatorname{GL}_d$ admits a decomposition $A = R \Sigma$ with $(R,\Sigma) \in O_d(\R) \times S_d^{++}$. In our framework, we restrict the orthogonal component $O_d(\R)$ to rotations $ R \in SO_d(\R)$ and we constrain the symmetric positive-definite matrices to the subgroup of diagonal matrices with strictly positive entries. This last component defines the space of anisotropic scalings, which can be parameterized by a vector $\alpha\in\R^d_{>0}$ that acts on $\R^d$ via the diagonal matrix $D_{\alpha}=\operatorname{diag}(\alpha_1,...,\alpha_d)$:
\[
\D{\alpha} x = (\alpha_1x_1,\dots \alpha_dx_d)^{\top}
\]
Finally, we combine these group of deformations to define the extended group of anisotropic scalings and isometries
\[
\D{\alpha}\operatorname{-Isom}(\R^d) \coloneq \R_{>0}^d\times SO_d\times\R^d
\]
equipped with the direct product group law
\[
\left\{\begin{array}{l}
    (\alpha,R,T)\cdot(\alpha',R',T') = \left(\alpha\alpha',RR',T'+T\right)  \\
      (\alpha,R,T,)^{-1}=(\alpha^{-1},R^\top,-T).
\end{array} \right.
\]
\begin{remark} 
    The group of deformations $\D{\alpha}\operatorname{-Isom}(\R^d) $ is a direct group and not a semidirect group contrary to the group defined in \cref{sec:affine_motions}. Indeed, natural actions on $\R^d$ of the group of anisotropic scalings and the group of isometries are not directly compatible, as we will see.    
\end{remark}

A naive action of this group on $\R^d$ could be expressed as $ (\alpha,R,T)\cdot x = RD_\alpha x + T$. However, this fails to form a valid group action because the compatibility condition does not hold, that is $((\alpha,R,T)(\alpha',R',T')) \cdot x \ne (\alpha,R,T) \cdot ((\alpha',R',T') \cdot x)$. This failure arises from the non-commutativity of anisotropic scalings and rotation matrices, as both are basis-dependent. Rather, this group naturally acts on an augmented space $\R^R\times SO_d$ that keeps track of the orthogonal basis :
\[
 (\alpha,R)\cdot (x,R_x) = (RR_xD_{\alpha}R_x^\top x,RR_x)
\]
where $x\in \R^d$, and $R_x\in SO_d$ encoding the orthogonal frame associated with $x$.
Following this idea, we construct a valid action of $D_{\alpha}-\operatorname{Isom}(\R^d)$ on an augmented space of images, as done in \cref{sec:aff_match_problem}. We assume that any image $I$ is also enriched with the given data of an intrinsic orthogonal basis, encoded by a matrix $R_I\in \operatorname{SO}_d$, and an intrinsic origin $O_I \in \R^d$. This provides a local coordinate frame attached to the image, in order to perform rigid and scaling deformations. Thus, the anisotropic scaling will stretch each axis of this basis, and the rotation will be centered around $O_I$. We therefore define the action of $(\alpha,R,T)$ on the augmented image $(I,R_I,O_I) \in (\mathcal{I},SO_d,\R^d)$, by
\begin{equation}
\label{rigid_action_I}
    (\alpha,R,T)\cdot (I,R_I, O_I)= (I',RR_I,O_I+T)
\end{equation}
where the transported image $I'$ is given by
\begin{equation}\label{eq:transport_image2}
I'(x) = \left\{\begin{array}{ll}
     I\left(R_I\D{\alpha^{-1}}R_I^\top R^\top (x-O_I-T) + O_I \right) &\text{if } R_I\D{\alpha^{-1}}R_I^\top R^\top (x-O_I-T) + O_I \in \Omega \\
     0 &\text{otherwise}
\end{array}
\right.    
\end{equation}

This still fails to define a valid group action as the image domain $\Omega \subset \R^d$ is not necessarily stable by the action of isometries and scaling, similarly to \cref{sec:affine_motions:action}. Consequently, some information may be lost if the deformations pushes the image beyond the boundaries of $\Omega$. However, given that the support of images is strictly contained within $\Omega$ we retrieve \cref{prop:transport_support} giving that we can affinely transport images without loss of information. 

\begin{proposition}[Action property]
    Let $(I,R_I,O_I)\in \mathcal{I}\times SO_d\times\R^d$, and denote by $\operatorname{Supp}(I)\subset \Omega$ its support. Let $(\alpha,R,T),(\alpha',R',T')\in \D{\alpha}\operatorname{-Isom}(\R^d)$, and define 
    \[
    \begin{array}{l}
         S_1 = R'R_I\D{\alpha'}R_I^\top(\operatorname{Supp}(I)-O_I) +O_I+T'\\
         S_2 =  RR'R_I\D{\alpha}R_I^\top R'^\top(S_1-O_I-T') +O_I+T'+T  \\
    \end{array}
    \]
    We also suppose that 
    \begin{equation}
    \label{action_cond}
        S_1,S_2\subset \Omega.
    \end{equation}
    Then we get that
    \begin{equation}
    \label{action_prop}
    (\alpha,R,T)\cdot\left((\alpha',R',T')\cdot(I,R_I,O_I) \right)= (\alpha\alpha',RR',T'+T)\cdot (I,R_I,O_I)
    \end{equation}
    and the set $S_2$ is exactly the support of the transported image.
\end{proposition}

\begin{proof}
    As in \cref{prop:transport_support}, the condition \eqref{action_cond} is exactly the necessary condition to get a complete rigid transport of the image $I$ without any loss of information. We denote $I'\in \mathcal{I}$ the image such that 
    \[
    (\alpha',R',T')\cdot (I,R_I,O_I) = (I',R'R_I,O_I+T').
    \]
    and $I''\in\mathcal{I}$ such that 
    \[
    (\alpha,R,T)\cdot (I',R'R_I,O_I+T') = (I'',RR'R_I,O_I+T'+T).
    \]
    Let $z\in \Omega$, and $x,y\in \R^d$ such that 
    \[
    \left\{\begin{array}{l}
    z=RR'R_I\D{\alpha}R_I^\top R'^\top(y-O_I-T')+O_I+T'+T \\
    y=R'R_I\D{\alpha'}R_I^\top(x-O_I)+O_I+T'.
    \end{array}\right.
    \]
    In particular we get that $x\in \operatorname{Supp}(I)$ if and only if $y\in S_1$, and by hypothesis \eqref{action_cond},
    \[
    I'(y) = I(x)
    \] which also means that $S_1=\operatorname{Supp}(I')$. 
    Similarly, $z\in S_2$ if and only if $y\in S_1=\operatorname{Supp}(I')$, if and only if $x\in \operatorname{Supp}(I)$, and therefore
    \[
    I''(z)=I'(y)=I(x)
    \]
    This also means in particular that $S_2=\operatorname{Supp}(I^{\prime\prime})$. Finally, we have 
    \begin{align*}
        z &= RR'R_I\D{\alpha}R_I^\top R'^\top(y-O_I-T')+O_I+T'+T \\
        &= RR'R_I\D{\alpha}R_I^\top R'^\top\left(\left(R'R_I\D{\alpha'}R_I^\top(x-O_I)+O_I+T'\right) -O_I-T'\right)+O_I+T'+T \\
        &= RR'R_I\D{\alpha}R_I^\top R'^\top\left(R'R_I\D{\alpha'}R_I^\top(x-O_I)\right)+O_I+T'+T\\
        &= RR'R_I\D{\alpha\alpha'}R_I^\top (x-O_I) + O_I + T' + T.
    \end{align*}
   so that we get
   \[
   (\alpha\alpha',RR',T'+T)\cdot (I, R_I,O_I) = (I'', RR'R_I,O_I+T'+T).
   \]
\end{proof}
To enrich this model with diffeomorphic deformation, we cannot consider the semidirect product $\D{\alpha}\operatorname{-Isom}(\R^d)\ltimes \operatorname{Diff}_{C_0^k}(\R^d)$ since, as mentioned earlier, there is no action of $\D{\alpha}\operatorname{-Isom}(\R^d)$ on $\R^d$ and so no conjugation operation for the semidirect group law. Therefore, we consider the direct product group:
\[
\D{\alpha}\operatorname{-Isom}(\R^d)\times \operatorname{Diff}_{C_0^k}(\R^d)
\]
equipped with group law
\[
\left\{\begin{array}{l}
    (\alpha,R,T,\varphi)\cdot(\alpha',R',T',\varphi') = \left(\alpha\alpha',RR',T'+T,\varphi\circ\varphi'\right)  \\
      (\alpha,R,T,\varphi)^{-1}=(\alpha^{-1},R^\top,-T,\varphi^{-1}).
\end{array} \right.
\]
This group acts on the augmented space of images $\mathcal{I} \times SO_d\times \R^d$ via
\begin{equation}
    (\alpha,R,T,\varphi)\cdot (R_I,O_I,I) = (RR_I,O_I+T,I\circ\varphi^{-1}).
\end{equation}
Therefore, the evolution equation of the template image $(R_{I_S},O_{I_S},I_S)$ is deduced from the infinitesimal action
\begin{equation}
        \left\{\begin{array}{l}
    \dot{R}_t=A_tR_t \\
    \dot{O}_t = \beta_t \\
    \dot{I}_t = - \langle \nabla I_t,v_t\rangle\\
    (R_0,O_0,I_0)=(R_{I_S},O_{I_S},I_S)
    \end{array}\right.
\end{equation}
where $(A,\beta,v) \in L^2([0,1],\operatorname{Skew}_d \oplus \R^d \oplus V)$.

Observe that the image is deformed exclusively by the diffeomorphism. Indeed, if we wanted every deformation mode to act on the image, the natural action would be:
\begin{equation}
\label{eq:joint_action_total}
(\alpha,R,T, \varphi) \cdot ( R_I, O_I,I) = (RR_I, O_I+T,\tilde{I})
\end{equation}
where the transported image $\tilde{I}$ evaluated at a spatial point $x \in \mathbb{R}^d$ is given by:
\begin{equation}
\tilde{I}(x) =
\begin{cases}
I \bigg( \varphi^{-1} \Big( R_I D_\alpha^{-1} R_I^\top R^\top (x-O_I-T) + O_I \Big) \bigg) & \text{if } \varphi^{-1}(\dots) \in \Omega \\
0 & \text{otherwise.}
\end{cases}
\end{equation}
However, this does not define a group action, but rather a groupoid action, since it depends on the intrinsic spatial information of the image $(R_I,O_I)$.

\subsubsection{Matching problem and Hamiltonian equations}

The variational problem associated with the registration of the template image $I_S$ under the action of $\D{\alpha}\operatorname{-Isom}(\R^d)\times \operatorname{Diff}_{C_0^k}(\R^d)$ is next given by:
\begin{eqnarray}
    \label{eq:energy_image}
\inf_{(s,A,\beta,v)} J(s,A,\beta,v) 
&=& \frac{1}{2} \int_0^1  \vert s_t\vert^2 + \vert A_t \vert^2 + \vert \beta_t \vert^2 +  \vert v_t \vert_V^2 \, dt 
+ \mathcal{D}^{DA}_{\gamma}(\alpha_1, R_1R_{I_S},T_{I_S}+T_1,I_1) \\
     \text{ s.t }& &     \left\{
        \begin{array}{l}
\dot{\alpha}_t=s_t\alpha_t \\
\dot{R}_t = A_t R_t \\
\dot{T}_t = \beta_t \\
\dot{I}_t = -\langle\nabla I_t, v_t\rangle        
        \end{array} 
        \right.\notag
\end{eqnarray}
where $I_1 = I_S \circ \varphi_1^{-1}$ is the template image deformed by the diffeomorphism. To recover the total deformation, we apply the anisotropic scaling and isometry components to $I_1$. This yields the final deformed image $ \tilde{I}_1 = (\alpha_1,R_1,T_1) \cdot I_1$, using the action defined in \cref{rigid_action_I}. Consequently, the anisotropic scalings and isometries do not deform the template image during the flow. Instead, they are optimized via the dissimilarity term:
\begin{equation}
\mathcal{D}^{DA}_\gamma(\alpha,R,T,I)=\gamma \Vert (\alpha,R,T)\cdot I_S - I_T\Vert_{L^2}^2 + (1 -\gamma) \Vert (\alpha,R,O)\cdot  (I_S \circ \varphi^{-1}) - I_T\Vert_{L^2}^2 
\label{eq:attachment_term_decoupled}
\end{equation}
with $0 \leq \gamma \leq 1$, where we consider the action \eqref{rigid_action_I} of $(\alpha,R,T)$ on the augmented images $(I_S, R_{I_S},O_{I_S})$ and $(I, R_{I_S},O_{I_S})$. The first term evaluates the alignment achieved by the scaling and isometries only and the second term evaluates the alignement resulting from the total deformation. 

In the same fashion as done in \cref{sec:aff_match_problem}, we associate an Hamiltonian to this variational problem
\begin{align}\label{eq:hamil_decoupled}
H(\alpha,R,T,I,p^s,p^A,p^b,p^I,s,A,\beta,v) =& (p^s\mid s\alpha) + (p^A\mid AR) + (p^b\mid \beta) + \left(p^I\mid \langle-\nabla I,v\rangle\right) \\
&- \frac{1}{2}\left( \Vert s\Vert^2 + \Vert A\Vert^2 + \Vert \beta \Vert^2+\Vert v\Vert_V^2 \right) \notag
\end{align}
so that the critical points of the functional satisfy the Hamiltonian equations.

\begin{proposition}\label{prop:geod_eq_scal_isom}
    The geodesic equations associated with the variational problem \eqref{eq:energy_image} are
\begin{equation}
\label{eq:Hamiltonian_full_system}
\left\{
    \begin{aligned}
           &\dot{\alpha}_t = s_t\alpha_t \\
        &\dot{R}_t = A_tR_t \\
        &\dot{T}_t=\beta_t \\
         &\dot{{I}}_t = - \langle \nabla I_t, v_t  \rangle \\
        &\dot{p}^{s}_t=-s_t p^{s}_t \\
         &\dot{p}^A_t = A_t p^A_t\\
         &\dot{p}^b_t = 0 \\
        &\dot{p}^I_t = - \operatorname{div}(p^I_tv_t)
    \end{aligned}
\right. \quad \text{ where }
\quad
\left\{
    \begin{aligned}
       &s_t = p^{s}_t \alpha_t \\
        &A_t = p^A_t R^{\top}_t \\
        &\beta_t = p^\beta_t \\
        &v_t = -K_V(p^I_t \nabla I_t) 
    \end{aligned}
\right.
\end{equation}

\end{proposition}

\begin{remark}
    Because we rely on a direct product group structure, there are no cross-terms between the deformation modes in the resulting dynamics. In particular, within these geodesic equations, the image evolves exclusively under the action of the diffeomorphic flow. The affine components, as detailed earlier, do not act on the image during the integration but are instead accounted for via the data attachment term in \eqref{eq:attachment_term_decoupled}. Since the affine motion is inherently decoupled from the image dynamics in this formulation, there is no need to introduce an explicit change of variables as done in the previous section.
\end{remark}

\begin{proof}
    Let us consider the Hamiltonian $H$ defined in \cref{eq:hamil_decoupled}.
    The PMP states that the critical points satisfy the following Hamiltonian equations:
    \begin{equation}
    \left\{
        \begin{aligned}
            &(\dot{\alpha}_{t},\dot{R}_{t},\dot{T}_t,\dot{I}_t,\dot{p}^\alpha_t,\dot{p}^A_t,\dot{p}^b_t,\dot{p}^I_t)=(\partial_{p^s} H,\partial_{p^A} H,\partial_{p^b} H,\partial_{p^I} H,- \partial_{\alpha} H,- \partial_{R} H,- \partial_{T} H,- \partial_I H) \\
            &(\partial_s H , \partial_A H,\partial_\beta H,\partial_v H)=0 
        \end{aligned}
    \right.
    \end{equation}
    First, let us explicitly expand the terms composing the first part of the Hamiltonian: 
    \begin{align*}
        &(p^s \mid s\alpha) = \langle p^s, s\alpha \rangle \\
        &(p^A \mid AR) = \operatorname{tr}((p^A)^\top AR) \\
        &(p^b \mid \beta)=\langle p^b, \beta \rangle \\
        &\left(p^I\mid \langle-\nabla I,v\rangle\right) = - \int_{\Omega} p^I(x)\langle\nabla I(x),v(x)\rangle dx.
    \end{align*}

    Consequently, taking the partial derivatives of the Hamiltonian with respect to the control variables $(s,A,\beta,v)$ yields:
    \begin{equation}
    \left\{
        \begin{aligned}
            & \partial_s H = 0 \Longleftrightarrow s = p^s \alpha \\
            & \partial_A H = 0 \Longleftrightarrow A = p^A R^\top \\
            & \partial_{\beta}H = 0 \Longleftrightarrow \beta = p^b \\
            & \partial_{v}H = 0 \Longleftrightarrow v = -K_V(p^I\nabla I)
        \end{aligned}
    \right.
    \end{equation}
    which directly provides the equations for the optimal controls. 

    The differential equations for the covectors $p^s, p^A,$ and $p^b$ follow directly from the partial derivatives of $H$ with respect to the state variables. Since the translation $T$ does not explicitly appear in $H$, we directly get $\dot{p}^b_t = - \partial_T H = 0$. For the scaling parameter, we obtain $\dot{p}^\alpha_t = - \partial_\alpha H = -s_t p^s_t$. 
    For the rotation matrix, deriving the trace yields $\partial_R H = A^\top p^A$. However, since $A_t \in \operatorname{Skew}_d$, we have $A_t^\top = -A_t$. Therefore:
    \begin{equation*}
        \dot{p}^A_t = -\partial_R H = - A_t^\top p^A_t = A_t p^A_t.
    \end{equation*}

    We now focus on the equation for $p^I$. Using the product rule for the divergence operator, it follows that $\langle \nabla I , v p^I \rangle = \operatorname{div}(v p^I I) - I \operatorname{div} (v p^I)$. Then, the Hamiltonian equation corresponding to $p^I$ gives:
    \begin{align*}
        \dot{p}^I_t &= -\partial_{I} H \\
        &= \partial_I \int_{\Omega} p^I(x)\langle\nabla I(x),v(x)\rangle dx  \\
        &= \partial_I \int_\Omega \operatorname{div}(v p^I I) (x) dx - \partial_I \int_\Omega  I(x)\operatorname{div} (v p^I)(x) dx
    \end{align*}

    Since the image vanishes outside its support $\operatorname{Supp}(I) \subset \Omega$, the divergence theorem implies that $\int_\Omega \operatorname{div}(v p^I I)(x) dx =0$. Consequently, computing the derivative with respect to $I$ via the $L^2$ inner product identification yields:
    \begin{align*}
        \dot{p}^I_t &=  - \partial_I \int_\Omega  I(x) \operatorname{div} (v p^I)(x) dx \\
        &= - \partial_I \big(I \mid\operatorname{div}(v p^I)\big) \\
        &= -\operatorname{div}(p^I v) 
    \end{align*}
\end{proof}

In this section, we established the \textit{Full Affine} (\cref{sec:affine_motions}) and \textit{Decomposed Affine} (\cref{sec:decomp_affine}) models, integrating affine transformations, whether full or decomposed into isometries and scalings, with unstructured diffeomorphic deformations. While the derivation of the Hamiltonian equations provides the necessary mathematical framework to perform geodesic shooting for optimal joint registration, solving these systems in practice poses significant numerical challenges. In the next section, we present a numerical implementation of these models and evaluate their performance on both toy datasets and real medical images.

\section{Applications}
\subsection{Optimization strategy}

We now describe the practical implementation of the \textit{Full Affine} (FA) 
and \textit{Decomposed Affine} (DA) models introduced in the previous section.
Both follow a shooting approach: we search in the space of geodesics for the one 
bringing the template closest to the target. The core numerical task is 
integrating the geodesic equations from 
Propositions~\ref{prop:geod_eq_aff} and~\ref{prop:geod_eq_decoupled}, 
which couple the infinitesimal affine component with the LDDMM flow. From now on we consider that images are all discretised.

Our implementation builds on the \textsc{Demeter-Metamorphosis} library~\cite{franccois2021metamorphic, phd_AFrancois}, a PyTorch-based framework for LDDMM and metamorphosis, which we forked and extended to incorporate the affine coupling. Once the geodesics are integrated, gradients are obtained via automatic differentiation, avoiding the need to derive and implement adjoint equations by hand.
The code is available at~\footnote{\url{https://github.com/antonfrancois/Demeter\_metamorphosis/}}. 

A well-known difficulty in affine registration is the roughness of the
optimization landscape: the large number of degrees of freedom of affine
transforms generates numerous local minima, making gradient-based optimizers
prone to convergence to suboptimal solutions~\cite{jenkinson2001global,
jenkinson2002improved}. We address this with two complementary strategies:
\begin{enumerate}
  \item \textbf{Progressive affine enrichment.}
        We perform a coarse-to-fine search by gradually increasing
        the complexity of the affine component
        (e.g.\ translation only, then rigid, then full affine),
        providing a warm initialization at each stage.
  \item \textbf{Variational weighting.}
        We exploit the parameter $\gamma$ 
        from the data attachment terms~\eqref{eq:attachment_term_affine}
        and~\eqref{eq:attachment_term_decoupled} to assign greater
        weight to the affine component relative to the diffeomorphic
        one in the early stages of optimization,
        allowing the affine degrees of freedom to be estimated first.
\end{enumerate}

\paragraph{Covector Activation}

In our framework, registering two images reduces to the estimation of the covector associated with each component of the deformation. These consist of the vector-field covector $p^I : \Omega \to \mathbb{R}$, common to both models, and the affine covectors, which differ between them. In the FA model, the affine covector comprise a matrix $p^A \in \mathbb{R}^{d \times d}$ (with $d=2$ in 2D and $d=3$ in 3D) and a translation vector $p^\tau \in \mathbb{R}^d$. In the DA model, they consist of a rotation covector $p^R \in \mathbb{R}^l$, a translation covector $p^\tau \in \mathbb{R}^d$, and scaling covector $p^s \in \mathbb{R}^d$; when isotropy is assumed, all entries of $p^s$ are equal and reduce to a single scalar parameter. For any given set of covectors, we integrate the geodesic equations from Propositions~\ref{prop:geod_eq_decoupled} and \ref{prop:geod_eq_scal_isom} to obtain the deformed image and its associated deformation. 

A key advantage of this parameterization is that individual components can be frozen or activated simply by excluding the corresponding covector from the optimization. In PyTorch, this amounts to setting \texttt{requires\_grad = False} on the relevant tensors. This enables a progressive \emph{covector activation} strategy: one begins by optimizing only a restricted subset of covector — for instance, translation alone — and gradually unlocks additional components as the optimization proceeds. Fixing $p^I$ while optimizing the affine covector yields a pure affine registration; similarly, fixing $p^s$ within the DA model constrains the deformation to be rigid. 

In practice, the rotation is the most difficult component to recover, as the
energy landscape with respect to orientation contains many local
minima~\cite{jenkinson2001global}. We address this with a \emph{multi-start}
procedure: short optimizations are run from several candidate rotation
initializations, and the best solution is selected as the starting point for
the full joint optimization. For example, candidates can be chosen by uniformly sampling
rotation angles at a fixed increment, i.e.\ $\{2\pi k/n : k = 0,\ldots,n-1,\
n \in \mathbb{N}\}$. This lightweight strategy proves sufficient in our
experiments to avoid poor local minima without significant computational
overhead. For the FA model, each candidate angle is converted to its
corresponding rotation matrix before being fed to the optimizer. In both
models, the multi-start phase uses the data term~\eqref{eq:attachment_term_affine}
with $\gamma = 1$, disabling the diffeomorphic component entirely.








\paragraph{Parameter variational control \label{sec:var_control}}

The second strategy addresses a fundamental ordering problem in joint affine-diffeomorphic registration. Ideally, the affine component should account for the global, coarse differences between the moving and fixed images --- such as orientation, scale, and position --- while the diffeomorphic component refines the remaining local discrepancies. In practice, if both components are optimized simultaneously from the start, the diffeomorphic component tends to dominate. Because it is far more expressive, it can absorb large global differences by constructing a high-energy vector field, even when those differences would be more parsimoniously explained by a simple affine transformation. Once the diffeomorphic component has committed to such a solution early in the optimization, the affine component has little residual signal left to act on and remains near its initialization. The optimizer is then effectively stuck in a local minimum where the affine parameters are poorly estimated.

Concretely, we introduce a pseudo-time $k$ indexing the optimization iterations and define a composite energy:
\begin{equation*}
    E_k(p) = \gamma_k\, A(p) + (1 - \gamma_k)\, D(p),
\end{equation*}
where $A(p)$ measures the discrepancy between the target and the source deformed by the \emph{affine component alone}, $D(p)$ measures the discrepancy under the \emph{full} (affine and diffeomorphic) deformation, and $\gamma_k \in [0,1]$. When $\gamma_k = 1$ the optimizer attends exclusively to the affine alignment; as $\gamma_k \to 0$ it progressively shifts the focus to the joint registration. For notation ease, we will write $A_k := A(p_k)$ in the following.

We aim to control the relative contribution of each component to the total energy. By initially setting $\gamma$ to favour the affine term, we force the optimizer to first resolve the coarse geometric alignment before the diffeomorphic component is allowed to contribute meaningfully. As optimization proceeds, $\gamma$ is gradually relaxed, giving the diffeomorphic component increasing freedom to capture the finer, non-linear residuals. This progressive handover from affine to diffeomorphic deformation mirrors the coarse-to-fine philosophy common in multi-resolution registration schemes~\cite{jenkinson2002improved}, but controls the nature of the deformation model --- progressing from affine to diffeomorphic --- rather than the resolution of the image.

The key idea is to decrease $\gamma_k$ as the affine alignment stabilizes, i.e.\ as $\dot{A} \to 0$, while ensuring the transition is smooth enough to give the optimizer time to find the best compromise between the two components. To this end, we first define an affine rate estimator
\begin{equation}
    r_A(k) = \frac{1}{w - 1} \sum_{j=k-w+1}^{k-1} |A_j - A_{j-1}| \approx \overline{|\dot A|},
\end{equation}
where $w$ is the length of the trailing window over which increments are averaged.
We then define a sensitivity function $S$
\begin{equation*}
    S(k) = \begin{cases}
        1 & \text{if } k \leq w,\\
        \min\!\left(1,\, r_A(k)/\tilde r\right) & \text{otherwise,}
    \end{cases}
    \qquad \tilde r = \frac{|A_1 - A_0|}{2},
\end{equation*}
where $\tilde r$ normalizes the amplitude of $r_A$, making the remaining parameters dimensionless and easier to tune. Anchoring it on the first increment is deliberate: $|A_1 - A_0|$ measures how fast the affine term improves while it is still resolving the initial, coarse misalignment, and therefore scales with the magnitude of that misalignment. The threshold is thus adapted to each pair of images rather than fixed globally, and $S$ compares the current affine rate against the rate observed while the alignment was still coarse. 

Rather than responding only to affine improvement ($\dot{A} < 0$), $S$ responds to any large variation in $A$, whether positive or negative. A large positive $\dot{A}$ signals that the affine component is unstable --- potentially escaping a local minimum ---
and warrants keeping $\gamma$ high to maintain affine influence during that transition. In both cases, $|\dot{A}|$ being large indicates that the affine optimization is not yet settled, justifying a higher weight on the affine term. Conversely, when $\dot{A} \approx 0$ the affine has stabilized and $\gamma$ is allowed to decay, progressively handing control to the diffeomorphic component. 
The temporal evolution of $\gamma_k$ is then governed by
\begin{equation}
    \dot \gamma_k = \frac{2}{\nu}\gamma_k^\alpha  \left[S(k)(1 - \gamma_k)^\alpha - (1 - S(k))\right], \quad \gamma_0 =1; \quad \alpha = \frac12,
    \label{eq:dot_gamma}
\end{equation}

where $\nu > 0$ is a time constant, expressed in iterations, that lower-bounds the number of iterations required for $\gamma$ to reach zero. The evolution of $\gamma$ is designed to proceed in three phases: a waiting phase while $S(k) = 1$, a fast descent, and a smooth landing at zero. When $\gamma$ reach zero, only $D$ is optimized, we call this moment \textit{Handover}. 

\begin{wrapfigure}{R}{0.435\textwidth}
    \centering
    \includegraphics[width=\linewidth]{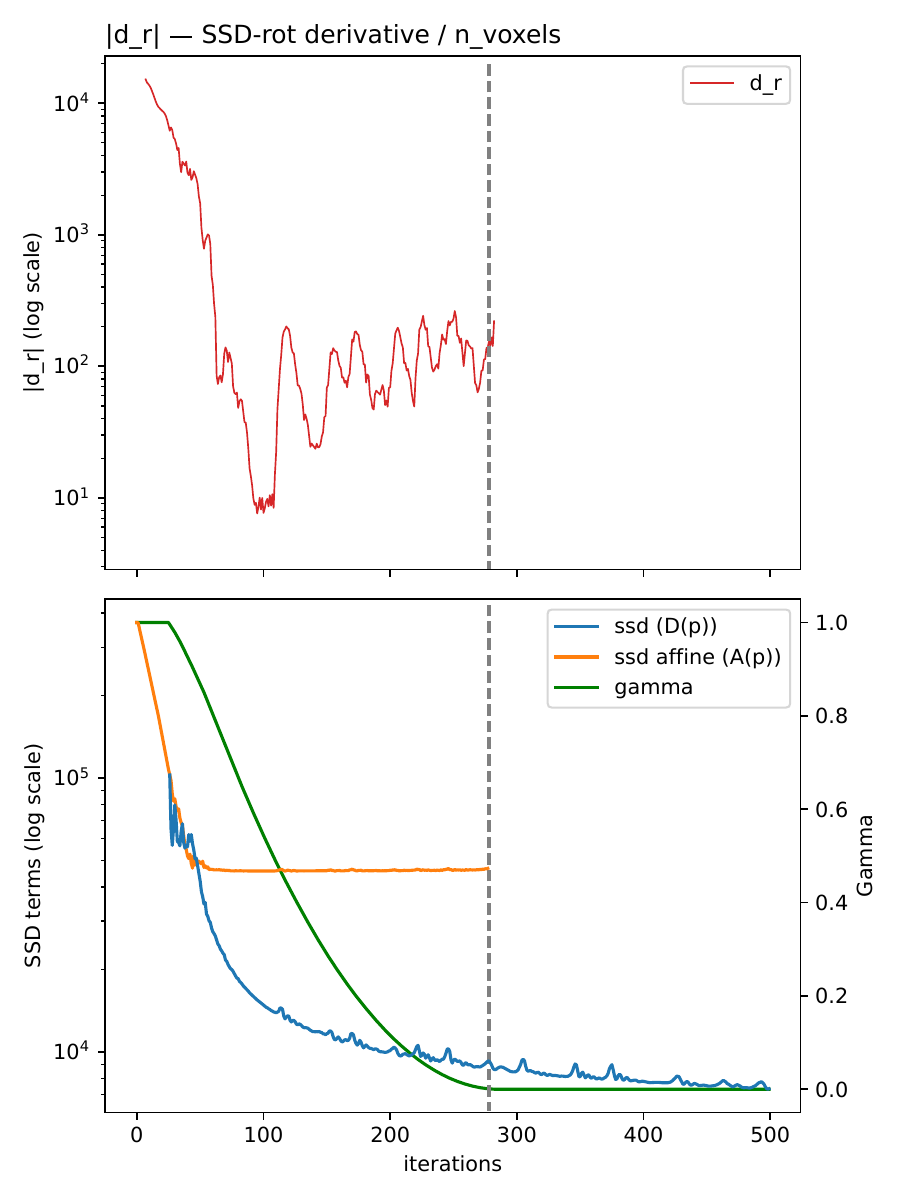}
    \caption{\textbf{Cost evolution  depending on $\gamma$.} 3D MRI image from the experiments presented in \cref{sec:med_img} and $w=8, \nu=250$ 
     $\tilde r$ activated at iter 26: $\tilde r=15191.03 $
     Handoff detected at iter 278 and shown by the gray dotted line.
    }
    \label{fig:variational_cost}
\end{wrapfigure}

\begin{remark}[fixed points]
The bracket in \cref{eq:dot_gamma} is strictly decreasing in $\gamma$, so its value at the origin, $2S-1$, governs the dynamics. The equation admits two fixed points, $\gamma^{*,1} = 0$ and $\gamma^{*,2} = \frac{2S-1}{S^2}$ (for $\alpha = \frac12$).
When $S \leq \frac12$, $\gamma^{*,2} \leq 0$ leaves $[0,1]$, so $\dot\gamma < 0$
throughout and $\gamma$ decays to zero; when $S > \frac12$, $\gamma^{*,2}$ is
attracting and increases to $1$ as $S \to 1$, holding $\gamma$ high while the affine is still moving. Handover is thus triggered once the affine rate falls below half its reference $\tilde r$, i.e.\ once $r_A(k) \leq |A_1 - A_0|/4$), rather than requiring $\dot A$ to vanish exactly.
\end{remark}
The parameters introduced above are few and largely interpretable; we now describe how each should be chosen.
\begin{enumerate}
    \item \textbf{The window $w$} averages the increments of $A$ and thereby controls the noise level of $S$; it also sets the length of the initial phase during which
    $S = 1$.

    \item \textbf{The exponent $\alpha$} controls the behavior near the origin. As
    $\alpha < 1$, the right-hand side is not Lipschitz at $\gamma = 0$, so $\gamma$
    reaches zero in finite time while $\dot\gamma \to 0$ ensures a smooth landing,
    whereas $\alpha = 1$ only approaches zero asymptotically. It is a structural
    choice, fixed to $\alpha = \frac12$.

    \item \textbf{The time constant $\nu$} sets the handover duration. Integrating  \cref{eq:dot_gamma} from $\gamma_0 = 1$ with $S = 0$ gives $t^* = \nu$, the fastest possible descent, so $\nu$ lower-bounds the iterations needed to reach zero, counted from the onset of the descent. It should be chosen so that $k_s + \nu$ remains below the iteration budget, $k_s$ being the iteration where  $S$ first crosses $\frac12$.
\end{enumerate}

\Cref{fig:variational_cost} illustrates the behavior of the scheme on a
representative subject from the IXI dataset presented in \cref{sec:med_img}, with $w = 8$ and $\nu = 250$. The affine rate $r_A$
(top panel, log scale) falls by more than two orders of magnitude over the first
fifty iterations as the affine component resolves the coarse misalignment, then
fluctuates around a residual floor without ever returning to its initial
amplitude. The sensitivity $S$ follows this decay and remains close to zero from
iteration $\sim\!90$ onwards, so that the bracket in \cref{eq:dot_gamma} stays
negative and $\gamma$ decreases monotonically. Held at one during the warm-up,
$\gamma$ then descends smoothly and reaches zero at iteration $\sim\!278$, consistent with the bound $k_s + \nu$ of the preceding remark; the remaining
iterations are devoted entirely to the diffeomorphic component.

The bottom panel shows the corresponding data terms. The affine discrepancy
$A(p)$ drops steeply while $\gamma$ is still large and plateaus from iteration
$\sim\!50$, confirming that the affine component has settled well before control
is handed over. The full discrepancy $D(p)$ takes over from that point and keeps
decreasing long after $\gamma$ has vanished, capturing the non-linear residual
that an affine transformation cannot explain. The apparent collapse of $A(p)$ to
zero at iteration $\sim\!180$ is not a convergence effect: once
$\gamma < \varepsilon$, the affine-only term no longer contributes to the energy
and is therefore not evaluated.

This strategy tries adapt to the specific dynamics of a registration. Other strategies could be developed. The key insight of this paragraph is that: \textbf{If we perform join registration, there has still to be a coarse-affine to fine-diffeomorphic handover and thus we need a control on the optimisation trajectory.}

\paragraph{Optimizer choice}
Finally, this strategy is used in conjunction with the Adam optimizer~\cite{kingma2017adam}, which offers two key advantages in our setting: it provides better control over the maximum step size, and crucially, it allows different learning rates to be assigned to each set of covector. This per-component control was essential to achieving proper convergence, as the affine and diffeomorphic covector
operate on very different scales.

\subsection{FA and DA models on toy examples}

\begin{figure}[htbp]

    \begin{subfigure}[t]{.2\textwidth}
        \centering
        \includegraphics[width=.9\linewidth]{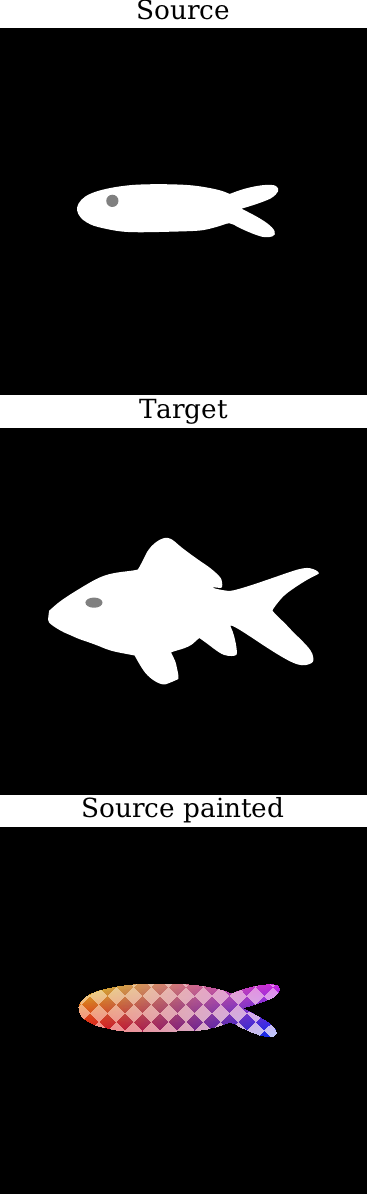}
        \caption{%
            \textbf{Source, target, and source texture.}
        }
        \label{fig:toyExample_source_target}
    \end{subfigure}%
    \begin{subfigure}[t]{.75\textwidth}
        \centering
        \includegraphics[width=\linewidth]{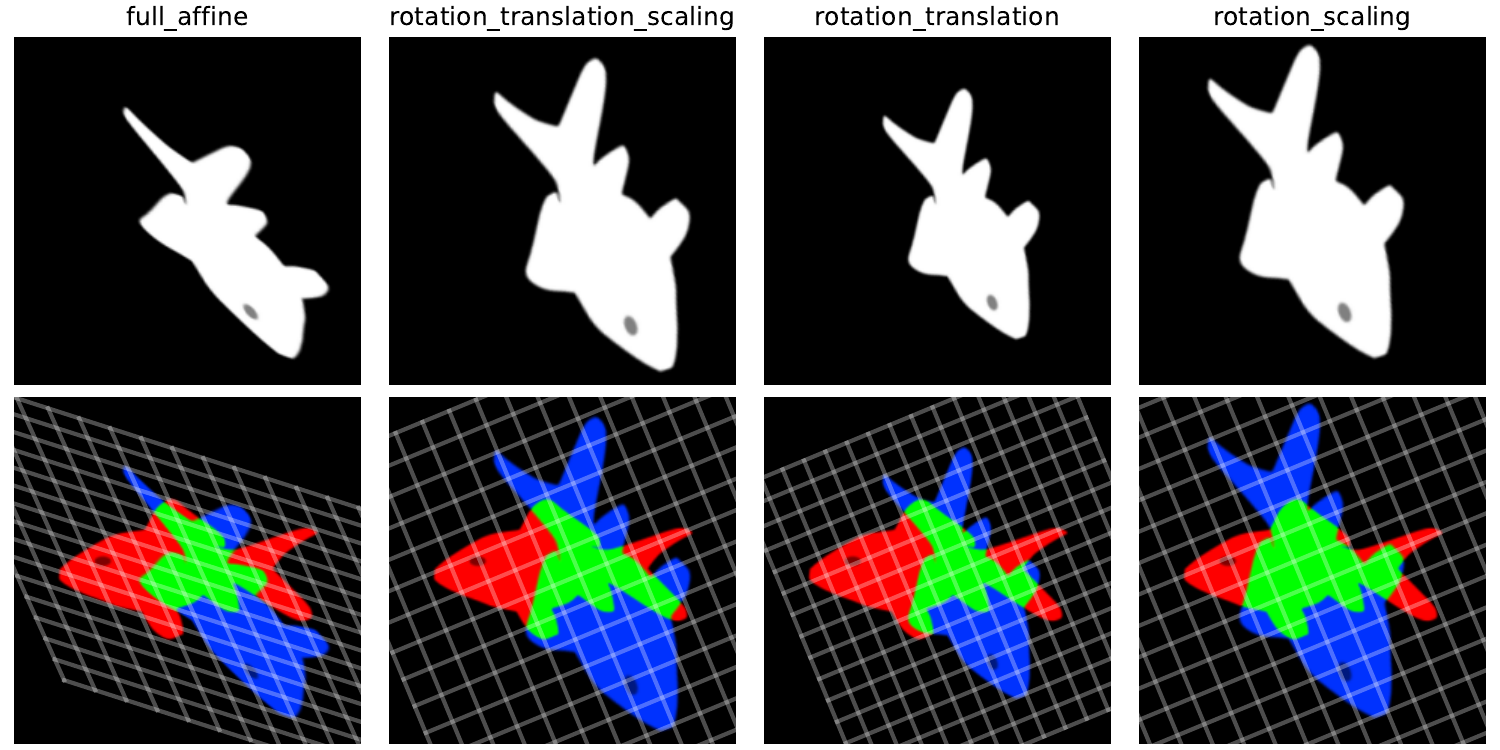}
        \caption{%
            \textbf{Target shapes.}%
        }
        \label{fig:toyExample_targets}
    \end{subfigure}

    \vspace{1ex}

    \begin{subfigure}[b]{\textwidth}
        \centering
        \includegraphics[width=\linewidth]{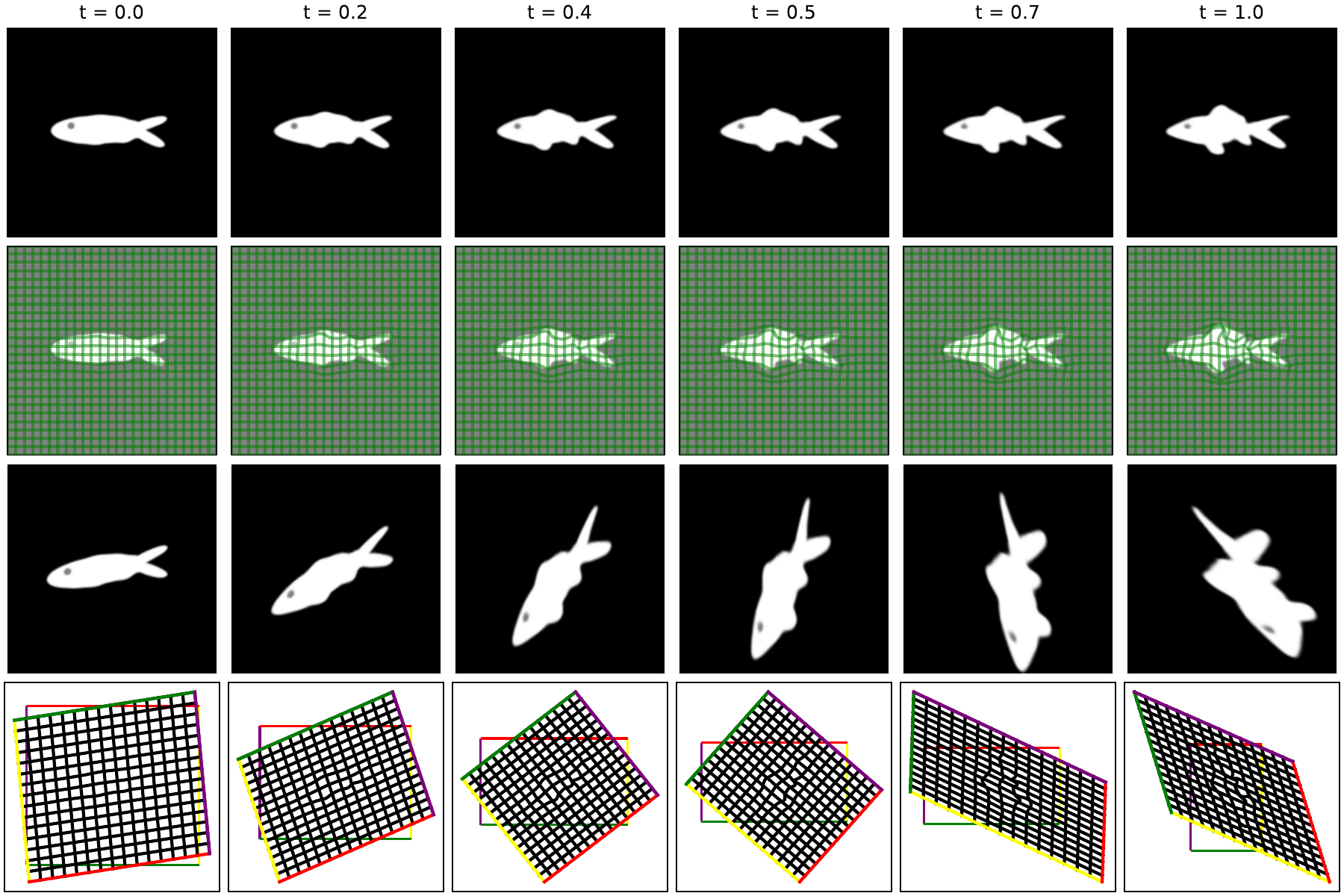}
        \caption{Fish integration.}
        \label{fig:toyExample_integration}
    \end{subfigure}

    \caption{\textbf{Toy example: fish registration overview.}
    \textbf{(a)}~Simple elongated fish source registered to a larger goldfish target.
    The source is painted with a checkerboard texture (\emph{painted source}) to visualize
    the deformation field throughout registration.
    \textbf{(b)}~Four target shapes derived from the raw target via constrained affine
    transformations: full affine~(A), R+S+T, R+T, and R+S, defining the four columns
    of \cref{fig:fish_grid_along} and \cref{fig:fish_grid_succ}.
    \textbf{(c)}~Integration on the full affine target using method FA:
    top two rows show the diffeomorphic evolution, bottom two the total deformation where the affine evolution has been added.
    }
    \label{fig:toyExample_source_targets}
\end{figure}

\begin{figure}[htp]
  \begin{adjustbox}{addcode={\begin{minipage}{\width}}{\caption{%
    \textbf{FA-LDDMM and DA-LDDMM models: qualitative results on the four targets of \cref{fig:toyExample_targets}.}
    Each group of six images is organised as follows.
    \textit{Left:} full registration (affine\,+\,diffeomorphic component); top row shows
    the registered image, bottom row its overlay with the target.
    \textit{Center:} affine component only; same row organisation.
    \textit{Right:} diffeomorphic component only; top row shows the registered image,
    bottom row the painted source from \cref{fig:toyExample_source_target} transported
    by the diffeomorphism, revealing internal deformations.
      }\label{fig:fish_grid_along}\end{minipage}},rotate=270,center}
      \includegraphics[scale=.5]{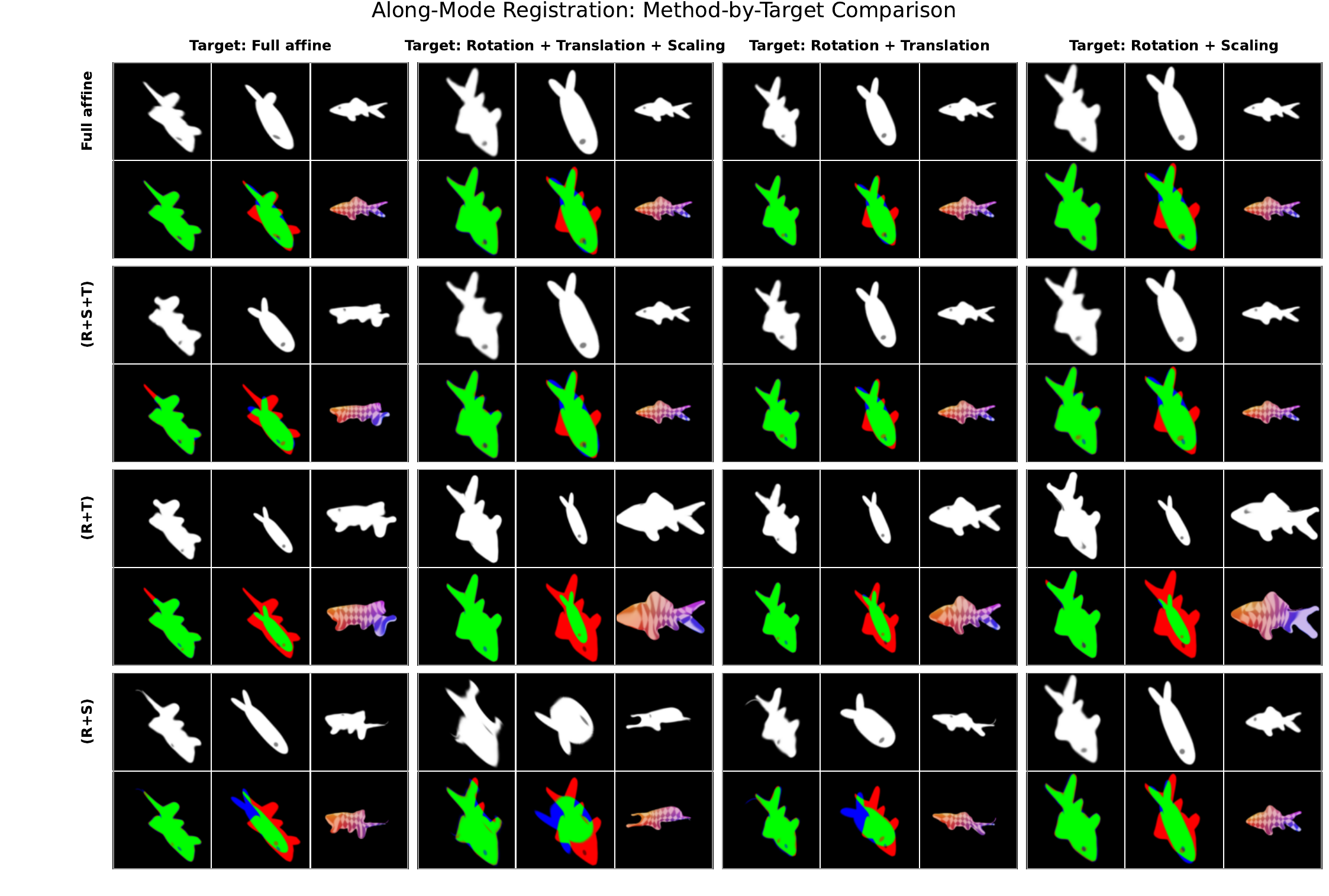}%
  \end{adjustbox}
\end{figure}

\begin{figure}[htp]
  \begin{adjustbox}{addcode={\begin{minipage}{\width}}{\caption{%
    \textbf{Affine registration followed by classical LDDMM: qualitative results on the four targets of \cref{fig:toyExample_targets}.}
    Each group of six images is organised as follows.
    \textit{Left:} full registration (affine\,+\,diffeomorphic component); top row shows
    the registered image, bottom row its overlay with the target.
    \textit{Center:} affine component only; same row organisation.
    \textit{Right:} diffeomorphic component only; top row diffeo grid as unlike  \cref{fig:fish_grid_along} we cannot isolate the diffeomorphism component,
    bottom row the painted source from \cref{fig:toyExample_source_target} transported
    by the diffeomorphism, revealing internal deformations.
      }\label{fig:fish_grid_succ}\end{minipage}},rotate=270,center}
      \includegraphics[scale=.4]{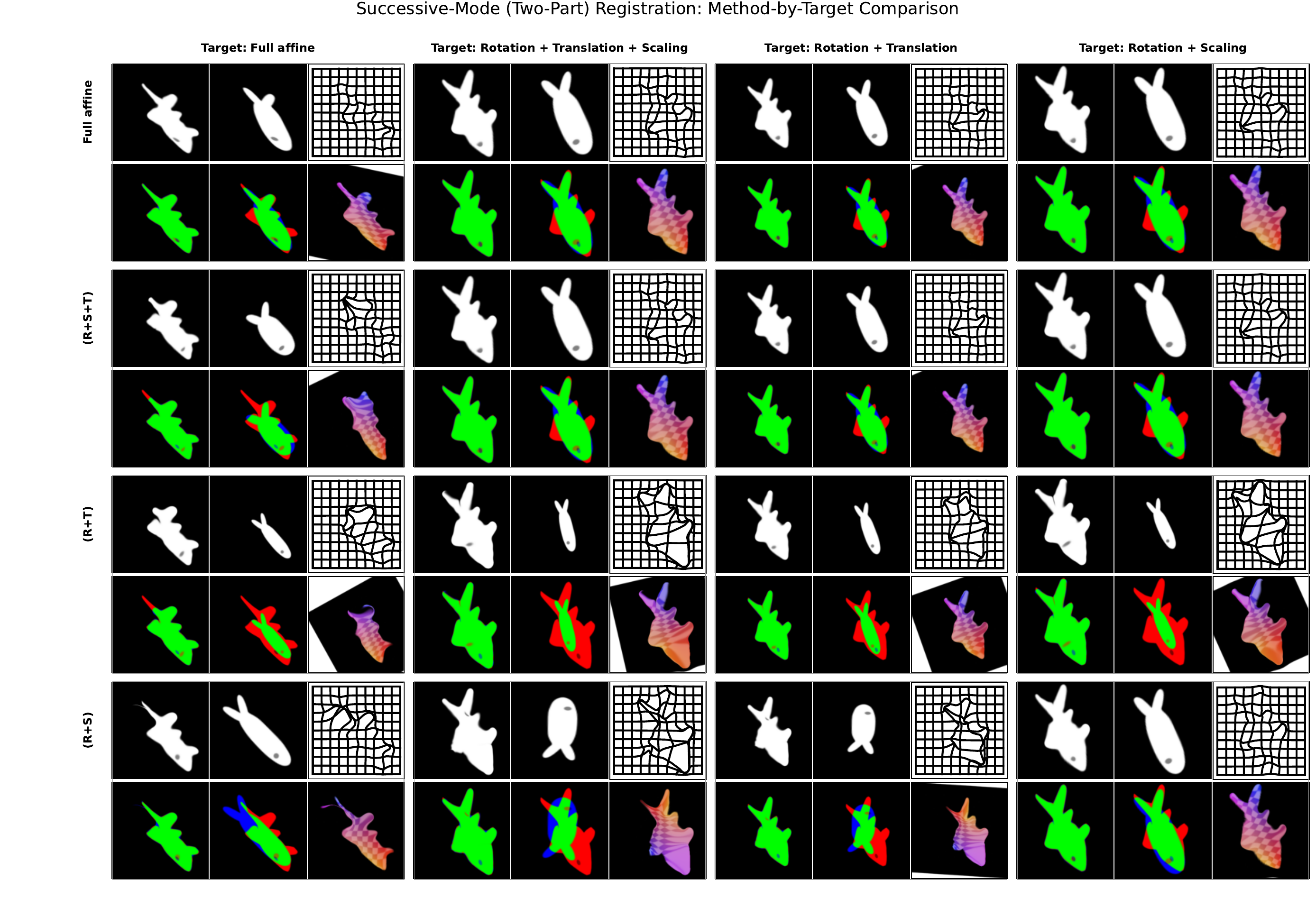}%
  \end{adjustbox}
\end{figure}

We demonstrate our approach on 2D toy examples designed to illustrate the behavior of each model and to assess their respective capabilities. \Cref{fig:toyExample_source_targets} gives an overview of the experimental setup that is displayed on \cref{fig:fish_grid_along} and \cref{fig:fish_grid_succ}.
For these experiments, we consider images taken in landscape orientation.

\Cref{fig:toyExample_source_target} shows the raw source and target images alongside the \emph{painted source}, a diagnostic visualization in which a checkerboard texture is overlaid on the source before registration. This tool is particularly revealing: a deformed shape may appear visually acceptable in terms of silhouette, yet the painted texture exposes local artefacts such as severe shearing or compression that betray a poor affine initialization or an over-reliance on the diffeomorphic component. We encourage the reader to attend closely to the checkerboard texture in the registered painted images throughout \cref{fig:fish_grid_along} and \cref{fig:fish_grid_succ}.
\Cref{fig:toyExample_targets} displays the four target images, each corresponding to a distinct mode of affine transformation: a full affine deformation, each individual component of the decomposed model (rotation, translation, scaling), and two combined modes (rotation-translation and rotation-scaling). This experimental design serves three purposes:
\begin{enumerate}
    \item to verify that the DA model behaves as expected when applied to each target, as examined in the registration grids of \cref{fig:fish_grid_along} and \cref{fig:fish_grid_succ};
    \item to provide a pedagogical illustration of the effect of each affine component in isolation and in combination;
    \item to compare our joint affine-diffeomorphic method (the \emph{along} mode; \textit{c.f.:} \cref{fig:fish_grid_along}) against the standard sequential approach used in the LDDMM literature, where affine registration is performed as a pre-processing step prior to the diffeomorphic optimization (the \emph{successive} mode; \textit{c.f.:} \cref{fig:fish_grid_succ}).
\end{enumerate}

In addition, \cref{fig:toyExample_integration} illustrates the geodesic integration trajectory for the FA model applied to the full affine target. The two top rows show the evolution of the diffeomorphism: the vector field is evaluated at the current position of the source, allowing it to capture shape differences precisely where they occur. The two bottom rows display the full registration trajectory, showing the combined affine and diffeomorphic deformation; the deformation grids represent the composition of both components at each step.

The registrations in \cref{fig:fish_grid_along} follow the variational control scheme of \cref{sec:var_control}: for each of the four methods (FA, R+S+T, R+T, R+S) applied to each of the four targets, we first run a brief affine-only search via progressive covector activation, then let the joint model converge to its final affine-plus-diffeomorphic geodesic. The registrations in \cref{fig:fish_grid_succ} instead follow the standard two-stage procedure. First, we register with the corresponding FA or DA implementation but with the diffeomorphic component deactivated, yielding a pure affine registration $(A,b)$. Second, we apply this transformation to the source image, $\tilde S = (A,b)\cdot S$, and run a classical LDDMM registration from $\tilde S$ to $T$.

First, a visually good final registration does not guarantee a well-behaved diffeomorphic deformation. In other words, in both figures, the LDDMM part registers the target shape effectively, but it is clear that the resulting deformation is far from a natural-looking one. This is particularly apparent in the R+T rows, where the registered shape appears reasonable but the deformation of the painted texture reveals large, irregular distortions. The deformation grids in the successive mode (\cref{fig:fish_grid_succ}) are particularly revealing: in the R+T rows (third row pair), the grids show not just large deformations but genuinely pathological ones — the grid lines cross each other, indicating that the diffeomorphic map is close to, or actually violating, injectivity in some regions. This is a strong qualitative indicator that the successive approach struggles when the affine pre-registration leaves the source far from the target. By contrast, in the along mode, the painted texture deformations, while imperfect, do not show this kind of fold-over behavior.

Second, the joined methods (\cref{fig:fish_grid_along}) generally produce smoother and better-distributed deformations than the successive ones, even when the two methods yield similar-looking affine pre-registrations. In the (R+T method, R+T target) case in particular, the along method centers the fish more effectively after the affine step, which reduces the burden on the diffeomorphic component and results in a cleaner overall deformation. In the Full affine row, the deformation grid in the successive mode (\cref{fig:fish_grid_succ}) is remarkably regular compared to other rows — nearly a flat grid with only mild distortion — across all four targets. This makes sense: when the affine component has full degrees of freedom, it can absorb most of the shape difference, leaving very little work for the diffeomorphic step. Interestingly, the joined method in the same configuration produces a similarly clean deformation, suggesting that the two approaches converge once the affine parameterization is rich enough.

The four rows of \cref{fig:fish_grid_along} and \cref{fig:fish_grid_succ} follow the same ordering as the four columns (full affine, R+S+T, R+T, R+S), so that diagonal cells pair each method with its matching target. Third, registrations in the lower triangle of the grid — for example, the (R+S method, FA target) cell, where the affine transformation used by the method has fewer degrees of freedom than the one used to generate the target — are structurally underconstrained: the affine step cannot fully account for the required shape change. Despite this, the diffeomorphic component still manages to partially compensate, though at the cost of larger local deformations, as evidenced by the distorted grid in the successive mode.

Fourth, cells in the upper triangle — for example, the (FA method, R+S target) cell, where the method's affine parameterization has at least as many degrees of freedom as the target transformation — generally produce good registrations in both modes. Diagonal cells, where method and target coincide, are the most favorable case: for instance, the (R+S+T, R+S+T) pair converges to similar deformations in both the along and successive modes, for the same reason noted above for the Full affine row. A closer look, however, reveals that the along diagonal cells tend to produce slightly more compact and centered intermediate shapes than the successive ones, even in this favorable setting.

\subsection{Medical images \label{sec:med_img}}

\paragraph{Dataset and experiment overview }
We test our implementation on images from the IXI dataset\footnote{IXI brain dataset: \url{https://brain-development.org/ixi-dataset/}}, a collection of MR images from healthy subjects acquired at three hospitals in London. For this study, we use the T1-weighted images only. As is standard practice for brain MRI, we register all images to an atlas, the MNI template~\cite{fonov2009unbiased, fonov2011unbiased, collins1999animal}, constructed as the nonlinear average of 152 normal adult MRI scans.

To evaluate registration performance, we derive five anatomical structures from the FastSurfer~\cite{henschel2020fastsurfer} parcellation by grouping its labels into the following classes: subcortical gray matter (SCGM), cerebrospinal fluid (CSF), cortical gray matter (GM), white matter (WM), and brainstem. After registration, the resulting deformation is applied to the source segmentation, and Dice coefficients are computed for each class. We report the median Dice score across the five regions as an overall measure of anatomical correspondence. We randomly selected $54$ subjects from the dataset and visually inspected each segmentation to verify its quality; the cohort size is limited by this manual inspection step.

All registrations are performed in a fixed direction: each IXI volume acts as the moving image and the MNI template as the fixed image, so the target is common to all $54$ pairs. The target segmentation is obtained by applying FastSurfer to the template itself, so that source and target labels are produced by the same pipeline and the Dice scores are not affected by a discrepancy between labeling conventions. Preprocessing is deliberately minimal. Skull stripping is inherited from the FastSurfer segmentation and requires no separate step. MR images are accompanied by an affine matrix mapping voxel indices to scanner (world) coordinates; as is common in neuroimaging pipelines, we first apply this transformation to bring each subject into the orientation of the template, and resample the volume onto the template grid, of size $193 \times 229 \times 193$ at $1\,\mathrm{mm}$ isotropic resolution. Because the cohort is acquired on three different scanners and our similarity term is intensity based, we additionally normalize each volume by clipping intensities at their $99$th percentile and rescaling to $[0,1]$; the upper clip suppresses the sparse hyperintense voxels that would otherwise dominate the scaling. No bias field correction, histogram matching, or inter-site harmonization is applied. As illustrated in the source-vs-target panel of \cref{fig:ixi_brains}, the alignment obtained from the world-coordinate transformation alone remains coarse.

\paragraph{Competing methods}
As discussed in the introduction (\cref{sec:intro}), classical pipelines customarily perform the affine registration as a separate preprocessing step. For this step we use FLIRT~\cite{jenkinson2002improved}, as it gives the best Dice scores among the affine methods we tested. After applying FLIRT we apply LDDMM on the outcomes using the implementation from Demeter\_metamorphosis.
uniGradICON~\cite{tian2024unigradicon} is a foundation model for registration that produces a single warped grid without an explicit affine stage; the affine and diffeomorphic contributions are therefore entangled, and the affine part cannot be evaluated on its own. CARL~\cite{greer2025carl} is a very recent model that estimates the affine and diffeomorphic components jointly while keeping them separable, which lets us test our claim that joint registration yields better alignment. Both uniGradICON and CARL were run with instance optimization enabled (the IO variant), which makes the comparison with optimization-based classical methods fairer.

\paragraph{Implementation details.}
For our method, we use the variational scheme presented in \cref{sec:var_control}, first fitting the affine part only and progressively shifting weight toward the vector-field estimation, giving the scores \texttt{affine only} and \texttt{full}. We then apply a refinement phase that consists of applying a pure LDDMM on the found affine registration. 
We perform this refinement step because we noticed that Adam has trouble converging to the actual minimum due to difficulties controlling the gradient step, a problem resolved by LBFGS. However, as discussed in \cref{sec:var_control}, LBFGS is not the right optimizer choice for the joint registration. So both LDDMM steps after FLIRT and ours use LBFGS for fairness of comparison. 

Note that the refinement is not a return to sequential registration, because the affine it starts from was never estimated in isolation — it was selected jointly, to leave a residual a diffeomorphism can absorb cheaply. The comparison FLIRT+LDDMM vs. ours+refinement holds the LDDMM stage fixed and varies only the origin of the affine.
We want to emphasize that the reason refinement helps at all is an optimizer artefact (Adam step control vs. LBFGS), not a modeling deficiency.

The chosen reproducing kernel $K$ is a sum of Gaussians at scales $\sigma \in \{3, 7\}$, chosen to capture both the coarse displacement of the cortical mantle and the finer sulcal adjustment. The joint stage is optimized with Adam for at most $500$ iterations, stopping early once the relative change in total energy falls below $10^{-4}$ over three consecutive iterations. The variational control of \cref{sec:var_control} is run with $\nu = 250$ and $\alpha = 1/2$, the sensitivity $S$ being estimated from a trailing window of $8$ iterations. The refinement stage is a pure LDDMM initialized from the affine transformation recovered by the joint stage, optimized with LBFGS for at most $50$ outer iterations, each performing up to $20$ inner iterations with a history size of $10$, and stopping on a relative tolerance of $10^{-3}$ sustained over three iterations. The LDDMM stage applied after FLIRT uses the same LBFGS settings, so that the two pipelines differ only in the origin of the affine transformation.

\begin{figure}
    \centering
    \includegraphics[width=1\linewidth]{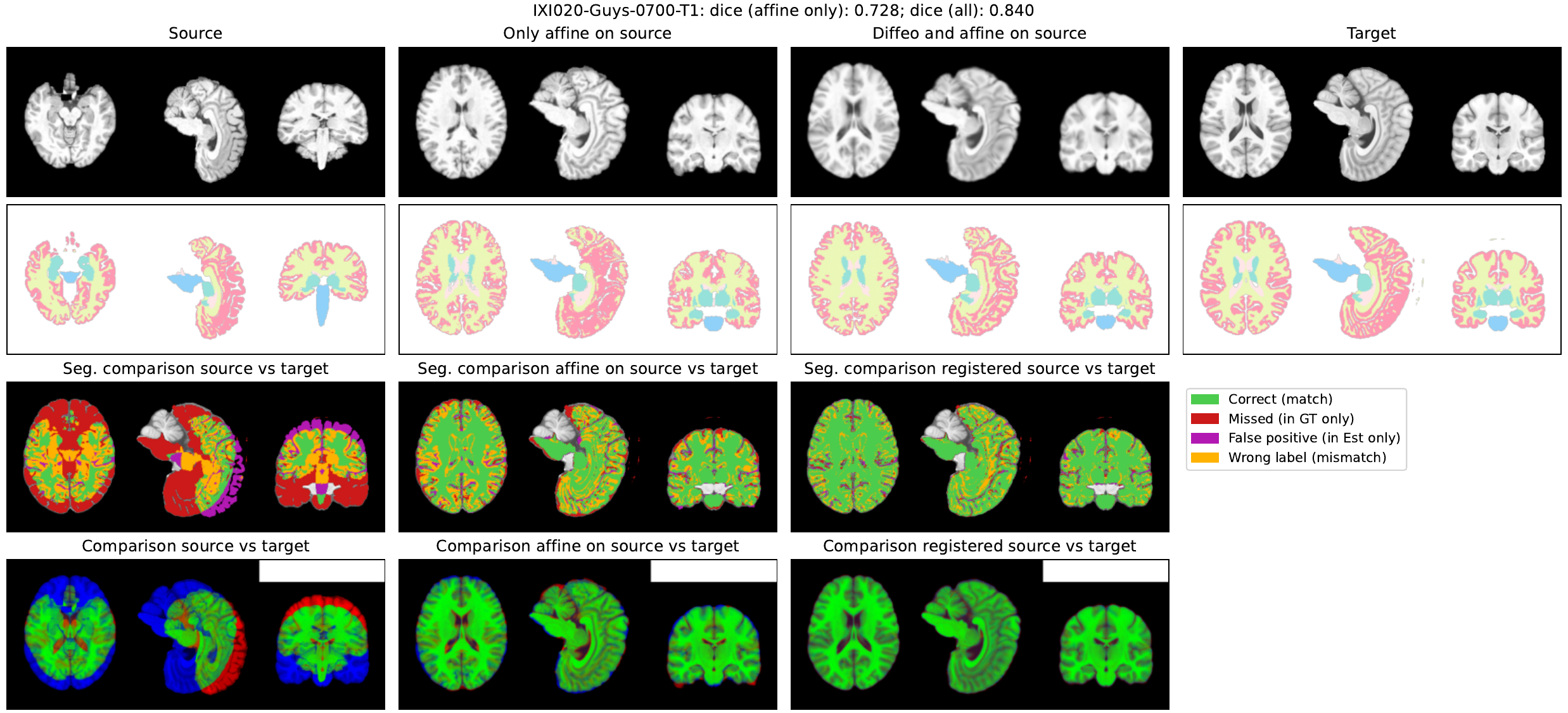}
    \caption{\textbf{Qualitative registration result on subject \texttt{IXI020-Guys-0700-T1} registered to the MNI template.} Each panel displays three orthogonal slices (axial, sagittal, coronal) taken through the centre of the volume. The target column shows the template.
    \textbf{Row~1 --- intensity:} source, source transported by the affine component alone, source transported by the full transformation (affine composed with the     diffeomorphism), and target.
    \textbf{Row~2 --- segmentations:} the five anatomical structures of the source, propagated by the same transformations, and the target labels.
    \textbf{Row~3 --- label agreement} with the target segmentation, before     registration, after the affine component alone, and after the full    transformation; green marks agreement, red labels present in the target only, purple labels present in the estimate only, and orange voxels assigned to the    wrong structure.
    \textbf{Row~4 --- intensity agreement} for the same three states, where green     indicates overlap and red and blue indicate intensity present in only one of the two images.
    Mean Dice over the five structures is $0.728$ after the affine component and $0.840$ after the full transformation; the refinement stage is not shown.
    }
    \label{fig:ixi_brains}
\end{figure}

\begin{figure}
    \centering
    \includegraphics[width=.6\linewidth]{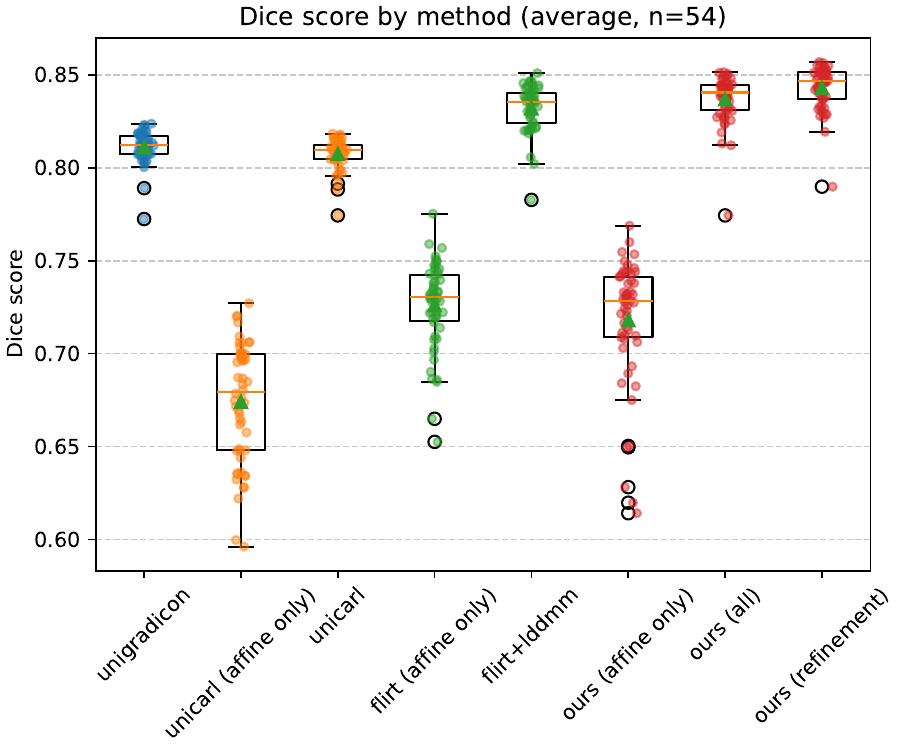}
    \caption{\textbf{Distribution of the mean Dice score over the five anatomical structures}, per method and deformation part ($n=54$ subjects). Boxes span the interquartile range, the central line marks the median, and whiskers extend to $1.5$ times the interquartile range. The distributions of FLIRT+LDDMM and of our method overlap substantially, which motivates the paired analysis of \cref{tab:dice_pairwise}.} 
    \label{fig:ixi_box_plot}
\end{figure}

\paragraph{Results}
\Cref{fig:ixi_brains} illustrates the behavior of the two components of the estimated transformation on a representative subject of the IXI cohort. The affine component alone already accounts for the bulk of the global misalignment: the residual red and blue bands that surround the brain in the leftmost comparison of the last row --- signatures of a global scale and pose mismatch --- are almost entirely removed, and the label agreement rises from a fragmented pattern to a largely green cortical ribbon, with a median Dice of $0.718$. What remains after this stage is local: the cortical folding pattern of the subject is still its own, and the disagreement concentrates in a thin band along the sulci and around the ventricular boundaries, where subject and template differ in shape rather than in position.

The diffeomorphic component resolves precisely this residual. The ventricles, whose outline is visibly thinner in the source than in the template, are brought into agreement, and the sulcal band narrows to a few voxels along the cortical surface, lifting the median Dice to $0.841$. The deep grey structures, already well placed by the affine component, are refined without being displaced. This division of labour --- global pose and scale absorbed by the affine part, local anatomical variability by the diffeomorphism --- is the intended behavior of the joint scheme, and is obtained here in a single optimization in which both components are estimated concurrently rather than in sequence.

\begin{table}[t!]                                                                                       
    \centering                                                                                         
    \caption{Dice score comparison by method (average over all structures, $n=54$). For comparison, before registration the average dice mean is $0.1934 \pm0.0707$  and median is $0.1708$}                  
    \label{tab:dice_comparison}                                                                       
    \begin{tabular}{llcc}                                                                              
      \toprule                                                                                         
      \textbf{Method} & \textbf{Deformation part} & \textbf{Mean $\pm$ Std} & \textbf{Median} \\       
      \midrule                                                    
        uniGradICON \cite{tian2024unigradicon}      & full          & $0.811 \pm 0.008$ & $0.812$ \\
      \midrule
      \multirow{2}{*}{CARL \cite{greer2025carl}}
                           & affine only   & $0.674 \pm 0.032$ & $0.679$ \\
                           & full          & $0.808 \pm 0.008$ & $0.810$ \\
      \midrule
      \multirow{2}{*}{FLIRT \cite{jenkinson2002improved} + LDDMM}
                           & affine only   & $\mathbf{0.726 \pm 0.023}$ & $\mathbf{0.731}$ \\
                           & full          & $0.832 \pm 0.012$ & $0.835$ \\
      \midrule
      \multirow{3}{*}{Full Affine LDDMM (ours)}
                           & affine only   & $0.718 \pm 0.034$ & $0.728$ \\
                           & full          & $0.837 \pm 0.013$ & $0.841$ \\
                           & full + refinement    & $\mathbf{0.843 \pm 0.012}$ & $\mathbf{0.847}$ \\
      \bottomrule                                                                                     
    \end{tabular}                                                                                          
\end{table}
\begin{table}[h!]
  \centering
  \small
  \begin{threeparttable}
  \caption{Paired comparisons on the IXI cohort ($n=54$ subjects). Each block fixes one arm of our method as the reference and compares it against every baseline sharing the same deformation part. \textbf{Wins}: subjects on which the reference scores higher; \textbf{Median $\Delta$}: median per-subject difference (reference $-$ comparator); \textbf{$r$}: matched-pairs rank-biserial effect size; \textbf{$p$}: one-sided Wilcoxon signed-rank ($H_1$: reference $>$ comparator). Holm--Bonferroni correction across the nine tests leaves all conclusions unchanged. The single non-significant test is the affine-only comparison against FLIRT; there the direction is in fact reversed, FLIRT attaining a higher affine-only Dice than ours ($p=2.3\times10^{-3}$ for the reverse hypothesis).}
  \label{tab:dice_pairwise}
  \begin{tabular}{ll ccc c}
    \toprule
    \textbf{Comparator} & \textbf{Part} & \textbf{Wins} & \textbf{Median $\Delta$} & \textbf{$r$} & \textbf{$p$} \\
    \midrule
    \multicolumn{6}{l}{\emph{Reference: ours, affine component only}} \\
    CARL~\cite{greer2025carl}                        & affine only & $48/54$ & $+0.045$ & $0.87$  & $1.6\times10^{-8}$ \\
    FLIRT~\cite{jenkinson2002improved}  & affine only & $16/54$ & $-0.003$ & $-0.44$ & $0.998$\tnote{$\dagger$} \\
    \midrule
    \multicolumn{6}{l}{\emph{Reference: ours, joint optimization}} \\
    uniGradICON~\cite{tian2024unigradicon}           & full & $53/54$ & $+0.029$ & $0.99$ & $1.0\times10^{-10}$ \\
    CARL~\cite{greer2025carl}                        & full & $53/54$ & $+0.033$ & $1.00$ & $8.6\times10^{-11}$ \\
    FLIRT~\cite{jenkinson2002improved} + LDDMM       & full & $49/54$ & $+0.006$ & $0.82$ & $7.7\times10^{-8}$ \\
    \midrule
    \multicolumn{6}{l}{\emph{Reference: ours, after refinement}} \\
    uniGradICON~\cite{tian2024unigradicon}           & full & $54/54$ & $+0.035$ & $1.00$ & $8.1\times10^{-11}$ \\
    CARL~\cite{greer2025carl}                        & full & $54/54$ & $+0.038$ & $1.00$ & $8.1\times10^{-11}$ \\
    FLIRT~\cite{jenkinson2002improved} + LDDMM       & full & $53/54$ & $+0.011$ & $0.99$ & $1.3\times10^{-10}$ \\
    Ours, joint optimization                         & full & $51/54$ & $+0.006$ & $0.98$ & $1.7\times10^{-10}$ \\
    \bottomrule
  \end{tabular}
  \begin{tablenotes}[flushleft]
    \footnotesize
    \item[$\dagger$] Direction reversed: the one-sided $p$ close to one indicates that FLIRT attains the higher affine-only Dice. The reverse hypothesis ($H_1$: FLIRT $>$ ours) gives $p = 2.3\times10^{-3}$.
  \end{tablenotes}
  \end{threeparttable}
\end{table}


\Cref{tab:dice_comparison} compares our method against current state-of-the-art techniques. Where the method permits, we evaluate the registration quality of the affine and diffeomorphic components separately.

The first observation from \cref{fig:ixi_box_plot} and \cref{tab:dice_comparison} is that on this task the classical pipelines, FLIRT+LDDMM and ours, outperform both learning-based methods we tested; for the affine component alone the comparison is possible only against CARL, since uniGradICON produces no separable affine. The margin between the two classical pipelines is narrow, and \cref{tab:dice_pairwise} is required to separate them. We use a one-sided Wilcoxon signed-rank test~\cite{rosner2006wilcoxon}, which tests whether the distribution of per-subject differences $(\mathrm{ours} - \mathrm{baseline})$ is centered above zero. Against uniGradICON and CARL the hypothesis holds with high confidence for every deformation part. Against FLIRT the picture is more nuanced. FLIRT attains a higher affine-only Dice than our affine component, and although the median difference is small, $-0.003$, it is significant ($p = 2.3\times10^{-3}$ for the reverse hypothesis, $r = -0.44$). Once the diffeomorphic component is composed, the ordering reverses: our joint result exceeds FLIRT+LDDMM on $49$ of $54$ subjects, and the refined result on $53$ of $54$.

We conclude that \textbf{the FA-LDDMM affine is a better starting point for diffeomorphic refinement than an affine estimated in isolation}. FLIRT minimizes an affine-only objective and, judged by that objective's own metric (DICE), does so better than we do. Our affine is not an affine-only optimum but the affine component of a joint affine and diffeomorphic optimum: it is selected to leave a residual that a diffeomorphism can absorb cheaply, not to minimize affine mismatch. That the two affine components are separated by a significant margin in FLIRT's favour, while the composed transformations are separated by a significant margin in ours, is direct evidence that affine-only overlap is a poor predictor of the quality of the final registration. Two affine transformations can achieve equal overlap and yet behave very differently under subsequent diffeomorphic refinement.

\section{Conclusion and future works}

In this article, we introduced a general framework based on an extension of LDDMM to jointly perform global affine transformations and local diffeomorphic deformations in image registration. The Full Affine (FA) model couples a general affine transformation with an LDDMM deformation, while the Decomposed Affine (DA) model separately parameterizes rotations, translations, and anisotropic scalings alongside the LDDMM deformation. Both models were implemented in two and three dimensions within the \texttt{Demeter} library. Unlike learning-based methods, our approach does not require training data or anatomical annotations. Moreover, the variational formulation makes it possible to compute continuous geodesic trajectories rather than a sequential composition of endpoint transformations. This provides an interpretable parameterization that explicitly distinguishes affine contributions from local non-linear components.

Numerical experiments demonstrated that our joint optimization strategy effectively avoids local minima and pathological deformations often encountered by standard affine-then-diffeomorphic procedures. To make this joint optimization tractable, we introduced a variational weighting strategy that manages the coarse-to-fine handover. This highlights that even in a joint formulation, it is important to control the optimization trajectory in order to prevent the diffeomorphic component from absorbing global affine motion. On brain MR images, our approach outperforms current state-of-the-art deep-learning and classical baselines. Indeed, a purely affine optimum is not necessarily a good initialization for diffeomorphic registration, underscoring the need for joint optimization of global and local motions. Note that while a final LBFGS refinement step was employed in our experiments, this does not revert to a sequential pipeline, since the affine initialization was derived from the joint optimization rather than estimated in isolation.

Yet, we should also mention a few avenues for further improvement and extension of the framework presented in this paper. While the computational cost per registration remains higher than the inference time of forward-pass learning methods, a direct comparison is asymmetric since our approach does not perform any training phase required by deep learning. Nevertheless, high-resolution 3D experiments still require a GPU with substantial VRAM. Furthermore, our medical imaging validation was restricted to a single anatomy (brain), a single target (MNI template), and a mono-modal similarity metric (SSD), although the \texttt{Demeter} library natively supports alternative metrics. Finally, the decoupling of the deformation modes remains partial. Indeed, our proposed method does not theoretically ensure that the diffeomorphic component cannot reproduce global affine deformations.

Future work will aim to enforce a stricter separation between affine and diffeomorphic components, either by constraining the diffeomorphic vector field to exclude affine modes or by penalizing its projection onto the space of affine deformations in the objective functional \cite{mouhli2026varifold}. To eliminate the need for a separate refinement stage, dynamic optimizer switching (e.g., from Adam to LBFGS) during the integration could also be explored. Other natural extensions include metamorphosis for intensity variations, multimodal similarity measures, and intermediate transformation groups such as locally affine models. Finally, retaining the full trajectory of the geodesics opens the way to more advanced statistical analyses on initial momenta, atlas construction, and parallel transport.

\clearpage
\addcontentsline{toc}{section}{References}
\printbibliography

\appendix

\section{Additional figures}

\begin{figure}
    \centering
    \includegraphics[width=1\linewidth]{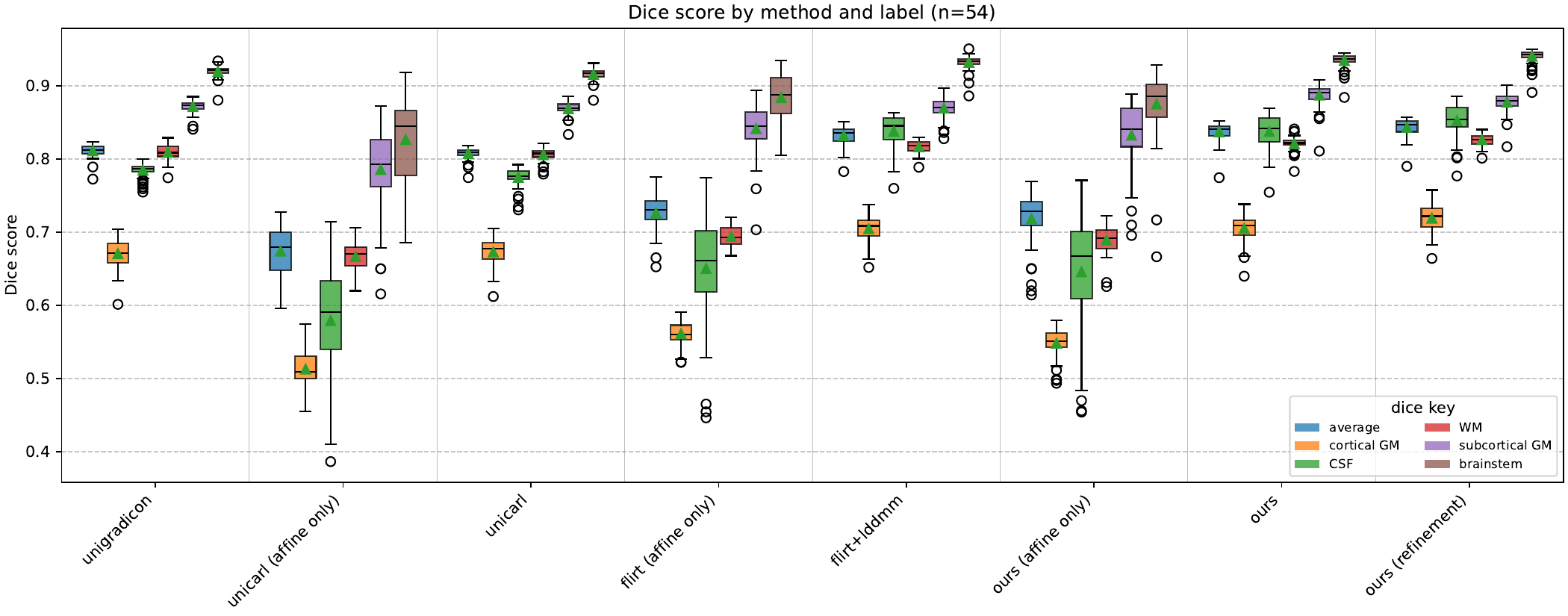}
\caption{                             
    Dice similarity coefficient per anatomical structure and registration method on the IXI test set ($n=54$ subjects, common to all methods). For each method, boxes show the distribution of Dice scores for the average over all structures (blue) and for each individual structure — CSF, cortical GM, subcortical GM, WM, and brainstem — following the segmentation label convention CSF~$=1$, cortical~GM~$=2$, subcortical~GM~$=3$, WM~$=4$, brainstem~$=5$. Box edges mark the 25th/75th percentiles, the horizontal line the median, the green triangle the mean, whiskers extend to $1.5\times$ the interquartile range, and circles denote outliers. It is the exhaustive version of figure \ref{fig:ixi_box_plot}.                           
  }                 
  \label{fig:dice_ixi_box_exhaustive}
\end{figure}

\appendix

\end{document}